\documentclass[12pt]{amsart}
\usepackage{wasysym}
\usepackage{amsfonts, amssymb, amsmath, amsthm}
\usepackage{amscd}
\usepackage{color}
\usepackage{enumerate}
\usepackage{tikz}
\usetikzlibrary{arrows}
\newcommand{\gp}{$G$-primitive}
\newcommand{\gps}{$G$-primitive }
\newcommand{\gpsperiod}{$G$-primitive.}
\newcommand{\gpscomma}{$G$-primitive,}

\newcommand{\cok}{\textnormal{cok}}
\newcommand{\blue}[1]{{\textcolor{blue}
{#1}}}
\newcommand{\red}[1]{{\textcolor{red}
    {#1}}}

\newcommand{\Conj}{\mathbb Z\textnormal{Conj}}
\newcommand{\SL}{\textnormal{SL}}
\newcommand{\EL}{\textnormal{El}}
\newcommand{\El}{\textnormal{El}}
\newcommand{\Z}{\mathbb{Z}}

\newcommand{\nkone}{\textnormal{NK}_1}
\newcommand{\ZG}{\mathbb Z G}
\newcommand{\N}{\mathbb{N}}
\newcommand{\R}{\mathbb{R}}
\newcommand{\rr}{\mathcal{R}}
\newcommand{\sss}{\mathcal{S}}

\newcommand{\nzc}{\textnormal{NZC}}
\newcommand{\nzcgpp}{\textnormal{NZC}(\Z_+G[t])}
\newcommand{\trace}{\textnormal{tr}}

\newcommand{\Rcal}{\mathcal{R}}

\newcommand{\tr}{\textnormal{tr}}
\newcommand{\gl}{\textnormal{GL}}
\newcommand{\el}{\textnormal{El}}

\newcommand{\As}{A^{\Box}}
\newcommand{\Bs}{B^{\Box}}
\newcommand{\mat}[4]{\left(\begin{array}{cc} #1 & #2 \\ #3 &
      #4 \end{array}\right)}

\DeclareMathOperator{\GL}{GL}

\DeclareMathOperator{\SK}{SK}
\DeclareMathOperator{\K}{K}

\newcommand{\annotation}[1]{\marginpar{\tiny #1}}

\theoremstyle{plain} 
\newtheorem{theorem}[equation]{Theorem}
\newtheorem{lemma}[equation]{Lemma}
\newtheorem{proposition}[equation]{Proposition}

\newtheorem{corollary}[equation]{Corollary}
\newtheorem{prop}[equation]{Proposition}

\theoremstyle{remark}
\newtheorem{question}[equation]{Question}
\newtheorem{pquestion}[equation]{Parry's Question}
\newtheorem{notation}[equation]{Notational Convention}
\newtheorem{remark}[equation]{Remark}
\newtheorem{definition}[equation]{Definition}
\newtheorem{erratum}[equation]{Erratum}
\newtheorem{example}[equation]{Example}

\newtheorem{realizationproblems}[equation]{Realization Problems}
\newtheorem{algebraicstudy}[equation]{Algebraic Study}
\newtheorem{sufficiencyofinvariants}[equation]{Sufficiency of invariants}

\newenvironment{enumeratei}{\begin{enumerate}[\upshape (i)]}%
                            {\end{enumerate}}
\input xy
\xyoption{all}
\numberwithin{equation}{section}
\begin{document}

\title[K-theory, Parry and  Liv\v sic]{Finite group extensions of shifts of finite type:
  K-theory, Parry and  Liv\v sic\\}

\author{Mike Boyle and Scott Schmieding}
\maketitle
\begin{abstract}
This paper extends and applies algebraic invariants
and constructions
for mixing finite group extensions of shifts of finite type.
For a finite abelian group $G$,
Parry showed how to define a $G$-extension $S_A$ from a
square matrix over $\Z_+G$, and classified the extensions
up to topological conjugacy by the strong shift equivalence
class of $A$ over $\Z_+G$.
Parry asked in this case if the dynamical zeta function $\det
(I-tA)^{-1}$
(which captures the ``periodic data'' of the extension) would
classify up to finitely many topological conjugacy classes
the extensions
by $G$ of a fixed mixing shift of finite
type.
When the
algebraic $\K$-theory   group $\textnormal{NK}_1(\ZG)$
is nontrivial (e.g., for $G=\Z/n$ with $n$ not squarefree)
and the mixing shift of finite type is not just a fixed point,
we show the dynamical
zeta function for any such extension
is consistent with infinitely many topological conjugacy classes.
Independent of $\textnormal{NK}_1(\ZG)$, for every
nontrivial abelian $G$ we  show there exists a
shift of finite type with an infinite family of
mixing nonconjugate $G$ extensions
with the same dynamical zeta function. We define computable complete
invariants for the periodic data of the extension for $G$ not
necessarily abelian, and extend all the above results to the
nonabelian case.  There is other work on basic invariants.
The constructions require the ``positive K-theory''
setting for positive equivalence of matrices over $\ZG[t]$.
\end{abstract}

\tableofcontents

\section{Introduction}
One part of the celebrated  paper \cite{Livsic}
of Liv\v{s}ic  shows that for certain
hyperbolic dynamical systems $T:X\to X$,
if the restrictions of
H\"{o}lder functions $f$ and $g$ to
the periodic points are cohomologous
as point set maps (i.e. ignoring topology), then they are
H\"{o}lder cohomologous
in $(X,T)$ --- i.e.,
$f=g+r\circ T -r$, with the transfer function $r$
being H\"{o}lder continuous.
(For an excellent introdiction to
the Liv\v{s}ic theory and to
cocycles in dynamical systems,
see  \cite{HasselblattKatokbook}.)
The proof of Liv\v{s}ic works for functions
into a metrizable abelian
group.
This result was generalized to
nonabelian groups for shifts of finite type by
Parry (see Remark \ref{ParryLivsicRemark})  and Schmidt
\cite{Parrylivsic,Schmidt1999},
and to more sophisticated
systems by various authors (e.g. \cite{Parrylivsic,Schmidt1999,Kalinin2011,Sadovskaya2013}).

Parry posed
a bold related question
in the case $G$ is finite abelian.  For
$(X,T)$ a mixing SFT and $f: X\to G$,
a suitable dynamical zeta function $\zeta_f$ encodes for all $n,g$
the number of periodic orbits of size $n$ and weight $g$.
Then $\zeta_f = \zeta_g$ if and only if
there is a bijection $\beta : \text{Per}(X)\to \text{Per}(X)$ such
that $f\circ \beta$ and $g$ are cohomologous
as point set maps. Parry asked, for $f: X\to G$ continuous and $G$ a finite
abelian group: does the set of continuous $g: X\to G$  with
$\zeta_g = \zeta_f$ contain only finitely many continuous
cohomology classes?  Parry's question probed  not only a possible
direction for extending  the
Liv\v{s}ic result, but also
the strength of conjugacy invariants for mixing SFTs
and their group extensions.
(The classification of cohomology classes of
functions from $X$ into a group is a
version of the classification of group extensions
of a system $(X,T)$.)
%
%

We will show that
for
many groups $G$ (the finite groups $G$ with $NK_{1}(\ZG) \ne 0$), the answer to
Parry's question is negative for {\it every}
nontrivial dynamical zeta function.
   The ingredients
for  this are the following.
\begin {enumerate}
\item
Generalizing the Williams' theory for SFTs,
Parry showed that any $G$-extension of an SFT $(X,S)$
can be presented
by a square matrix $A$ over $\Z_+G$, and two such group extensions
are isomorphic if and only if their presenting matrices are
strong shift equivalent (SSE) over the positive semiring $\Z_+G$ of
the integral group ring $\ZG$.
The dynamical zeta
function, with coefficient ring $\ZG$, is then $\zeta (z) = (\det (I-zA))^{-1}$.
Parry's theory, which he never
published, is presented in \cite{BS05} (in
Appendix \ref{leftvsright}, we correct an error in
the presentation in \cite{BS05}).
\footnote{The algebraic
invariants here over $\Z$ (shift and strong shift equivalence,
  $\det(I-tA)$), are parallelled
  in the study of shifts of finite type with Markov measure,
  where a finitely generated abelian group appears in place of the finite group
  $G$ \cite{MarcusTuncelwps,
    ParryTuncel1982stoch},
  and positivity issues around  $\det (I-tA)$ and shift equivalence
 become more analytic and formidable
\cite{Handelman2011}.}

\item
By Theorem \ref{sseclassif}, taken from \cite{BoSc1},
for any ring $\mathcal R$
and shift equivalence (SE) class $\mathcal C$ of matrices over $\mathcal R$,
the collection of
SSE classes over $\mathcal R$  of matrices in $\mathcal C$ is
in bijective correspondence with the group $\nkone(R)$
of algebraic K-theory.
If $\nkone(R)$ is not trivial, then it is not finitely generated as a
group
\cite{Farrell1977, WeibelBook}.  We give more background on
$\nkone(R)$ in Appendix \ref{sec:nk1}, and give
some concrete examples in Appendix \ref{sec:nk1}.
\item
In this paper, given $\nkone(\ZG )$ nontrivial,
we construct, for
any nontrivial mixing
SFT $(X,S)$,
infinitely many $G$-extensions  of $(X,S)$ defined by
matrices which pairwise
are SE over $\Z_+G$ but are not SSE over $\ZG$ (and
hence are not SSE over $\Z_+G$). Consequently,
these extensions pairwise are eventually conjugate; are not conjugate;
and have the same isomorphism class of conjugacy classes (in the
abelian case, this means they have the  same dynamical zeta
function). The construction arguments, carried out in Section
 \ref{parrysection}, use
constructive  tools available in the polynomial matrix
setting.
\end{enumerate}

In Section \ref{sec:pqa},
 we discuss Parry's question in
more detail, and we use the structure of shift equivalence of matrices
over $\ZG$ to address and clarify some other cases of
Parry's question (Sec. \ref{sec:pqa}).
We show that for every nontrivial finite group $G$,
there is an infinite collection of matrices
which are not SE-$\ZG$ and  which can be realized
in mixing extensions of SFTs with the
same periodic data.
Consequently, for every nontrivial finite abelian group $G$,
there is a dynamical zeta function compatible with
infinitely many SE-$\ZG$ classes which can be realized
in mixing extensions of SFTs.
On the other hand,
we give a class of mixing examples for which the dynamical
zeta function determines the SE-$\ZG$ class
(regardless of $\text{NK}_1(\ZG )$). For such a class,
known invariants do not provide a negative answer
to Parry's question. In no nontrivial case do  known constructions
provide a positive answer to Parry's question.

One purpose of this paper
is to summarize and extend our understanding of the
algebraic invariants for and approaches to mixing finite group extensions
of shifts of finite type (which we need
anyway for Parry's question).
(In particular, for not necessarily abelian finite groups $G$,
we give complete and computable invariants for the periodic
data of the $G$ extension of a shift of finite type.)
  There are two parallel formulations
for this. One involves SSE of matrices over $\ZG$ (Section
\ref{sec:fge}).
 The other formulation is in terms of the ``positive K-theory''
of polynomial matrix presentations (Section \ref{sec:pmp}).
In Appendix \ref{zgprimse}, we work out results
involving primitivity (some of which we need for proofs)
and shift equivalence to extend the theory parallel
to the theory over $\Z$.
In Appendix \ref{leftvsright} we review the basic connection of
matrices over $\Z_+G$ to $G$-extensions, and correct a
mistake in \cite{BS05}. (The mistake is only that the defining matrix should be associated to a left action of $G$, not a right action.)
Some open problems are listed in Section 7.

Mike Boyle is happy to acknowledge support
during this work
from the Pacific Institute for the
Mathematical Sciences and the
University of British Columbia. We thank 
 the referee for a 
  very careful and detailed  report, 
which has  notably sharpened the exposition.

\section{Finite group extensions of SFTs via matrices
over $\ZG$}
\label{sec:fge}
In this section we give basic definitions for finite group extensions;
describe the presentation of group extensions of SFT by matrices
over $\Z_+G$; and describe algebraic
invariants of defining matrices which correspond to
invariants of the group extensions.
Cocycles and the group extension construction
are
an important tool much more generally in dynamics (topological,
measurable and smooth), but
for simplicity, we restrict definitions to our special case.
We recommend
 \cite{HasselblattKatokbook}
for an introduction to cocycles in dynamics;
\cite{BS05} is a reference with proofs adapted
to some
of the items below, as indicated by references.

{\bf Standing assumption.} Unless indicated otherwise, from here $G$ denotes a finite group.
All $G$ actions are assumed to be continuous and
free unless indicated.

{\bf Basic definitions \cite{BS05}.}
Let a pair $(X,S)$ represent a homeomorphism
$S: X\to X$. We will be interested in only
two cases: either $(X,S)$ is a shift of finite
type, or it is a countable union of finite orbits,
with the discrete topology (i.e., we neglect topology).
A {\it group extension} of $(X,S)$ by $G$ is
a pair $(Y,T)$
together with a continuous map $\pi :(Y,T) \to (X,S)$
such that $S \pi  =  \pi T$; two points have the same image
under $\pi$ if and only if they are in the same $G$-orbit;
and $\pi$ is a covering map (for each point $x$ of $X$,
there is a neigborhood $V$ such that there are $|G|$
disjoint neighborhoods in $Y$ such that the restriction
of $\pi$ to each is a homeomorphism onto $V$).
If $(X,S)$ is SFT, then a $G$ extension of $(X,S)$ is a
free $G$-SFT, i.e.
an SFT $(X,S)$ together with a continuous free action of $G$
which commutes with the shift.
We will always take $G$ acting
{\it from the left}, for a correct matrix correspondence
in the case $G$ is nonabelian -- see Appendix \ref{leftvsright}
for an explanation, which corrects the choice
\lq\lq from the right\rq\rq
in \cite{BS05}.

Two $G$ extensions $(Y_1,T_1), (Y_2,T_2)$ are
{\it conjugate}, or {\it isomorphic}, if there is a
homeomorphism $\phi: Y_1\to Y_2$ such that
$\phi T_1 = T_2 \phi$ and
$\phi (gy) = g\phi(y)$
for all $y \in Y_1$.
Equivalently, they are isomorphic as $G$-SFTs.
A $G$ extension of $(X,S)$ may be constructed
from a continuous function $\tau : X\to G$
(a {\it skewing function}) as follows. Let
$Y= X\times G$ and define $T: Y\to Y$ by the
rule $(x,g) \mapsto (S(x), g \tau (x))$, with
$\pi : X\times G \to X$ the obvious map
$(x,g)\mapsto x$.  Every $G$-extension of
an SFT is isomorphic to one constructed in
this way, and for brevity we may refer to
such a group extension as $(X,S,\tau )$.

We say $G$-extensions
$(X_1,S_1,\tau_1 )$ and  $(X_2,S_2,\tau_2 )$
are {\it eventually conjugate} if   for all
but finitely many $n>0$ the $G$-extensions
$(X_1,(S_1)^n,\tau_1 )$ and  $(X_2,(S_2)^n,\tau_2 )$
are conjugate.

In a system $(X,S)$, continuous functions
$\tau_1$ and $\tau_2$ from $X$ to $G$ are
{\it cohomologous} if there is a continuous
function $\gamma : X\to G$ such that  for all $x$,
$\tau_1 (x) = (\gamma (x) )^{-1}(\tau_2 (x)) \gamma (Sx)$ in the group $G$.
 For $G$-extensions $(X_1,S_1,\tau_1 )$
and $(X_2,S_2,\tau_2)$, the following are equivalent:
\begin{enumerate}
\item
The two $G$-extensions are isomorphic.
\item
There is a homeomorphism $\phi: X_1\to X_2$
such that $\phi S_1 = S_2 \phi$ (i.e. $\phi$
is a topological conjugacy) and
the functions $\tau_2\circ \phi $ and $\tau_1 $
are cohomologous in $(X_1, S_1)$.
\end{enumerate}

A  {\it mixing} $G$-extension of $(X,S)$ is
a $G$-extension $(Y,T)$ of $(X,S)$ such that
$(Y,T)$ is topologically mixing. This is distinctly a
stronger assumption than the assumption that
$(X,S)$ is mixing. The mixing $G$-extensions are the
fundamental, central case.
(The papers \cite{akmfactor, akmgroup} of Adler, Kitchens
and Marcus describe invariants with which
the classification of some $G$ extensions of SFTs can be reduced
to this central case.)

{\bf Presentation by matrices over $\Z_+G$ \cite{BS05}.}
Suppose $A$ is a square matrix with entries in $\Z_+G$.
Then $A$ may be viewed as the adjacency matrix
of a labeled  directed graph, with adjacency matrix $\overline A$
defining  an edge SFT
$(X,S)$,
by setting
\begin{equation} \label{zgpres}
\tau (x)\ =\ \textnormal{the label of the edge }x_0 \ .
\end{equation}
Then
$(X,S,\tau)$
is a group extension of the SFT $(X,S)$.
Every group extension of an SFT is isomorphic to one
of this type.

{\bf Mixing.}
For an element $x=\sum_g n_g g$ of $\ZG$, we write
$x\gg 0$ if $n_g>0$ for every $g$, and say
$x$ is $G$-positive.
For a matrix $A$
over $\ZG$,
$A\gg 0$ means every entry is $\gg 0$.
We define a  \gps  matrix to be a square matrix over
$\Z_+G$ such that $A^{n}\gg 0$ for some $n>0$.

A nonzero square matrix $A$ contains a maximum principal
submatrix with no zero row and no zero column; this is the
{\it nondegenerate core} of $A$. For a property $P$,
a  matrix $A$ is {\it essentially} $P$ if its nondegenerate core
is $P$.
A matrix $A$ over $\Z_+G$ defines a mixing $G$-extension if and only if
it is essentially \gps  (Proposition \ref{gmixa}).

NOTE: The $\Z_+$ matrix $\overline A$ being primitive does
not guarantee that $A$ is primitive. (E.g., $A=(e+e)$ over
$\ZG$ with $G= \Z/2\Z$.)

{\bf Conjugacy and eventual conjugacy.}
$G$-extensions of SFTs presented by
matrices $A,B$ over $\Z_+G$ are conjugate if and
only if the matrices $A,B$ are strong shift equivalent
(SSE) over $\Z_+G$.
This theory, due to Parry and never
published by him, is presented in \cite{BS05}.
By Proposition \ref{eventualconjugacy}, these $G$-extensions
are eventually conjugate
if and only if $A,B$ are shift equivalent (SE) over
$\Z_+G$.
By Proposition \ref{primitiveeventual},
two  \gps  matrices are SE over
$\Z_+G$ if and only if they are SE over
$\ZG$.

{\bf Refinement of SE-ZG by SSE-ZG.}
For any ring $\rr$, the refinement of
SE-$\rr$ by SSE-$\rr$ is captured by the group
$\nkone (\rr )$ of
algebraic $\K$-theory, as follows.
\begin{theorem} \label{sseclassif} \cite{BoSc1}
Suppose $A$ is a square matrix over a ring $\rr$.
\begin{enumerate}
\item
If $B$ is SE over $\rr$ to $A$, then
there is a nilpotent matrix $N$ over $\rr$
such that  $B$ is SSE over $\rr$ to the
matrix
$A\oplus N=
\begin{pmatrix} A & 0 \\ 0 & N
\end{pmatrix} $.
\item
The map $[I-tN] \to [A \oplus N]_{SSE}$ induces a bijection from
$\nkone(\mathcal R)$ to the set of SSE classes of matrices
over $\rr$ which are in the SE-$\rr$ class of $A$.
\end{enumerate}
\end{theorem}
For more on $\nkone$, see Appendix \ref{sec:nk1}.

{\bf Periodic data and trace series.}
We consider $G$-extensions $(X,S,\tau )$ such that
$(X,S)$ has only finitely many orbits of size $n$,
and formulate ``periodic data'' which give a complete
invariant of isomorphism for the group extension
obtained by restriction of  $S$ and $\tau$ to
the periodic points of $S$, with the discrete
topology. (Caveat: in the context of a
Liv\v sic type theorem,
``periodic data''
may refer to the cohomology class of the
restriction of $\tau$ to the periodic points,
with discrete topology \cite{Sadovskaya2013}.
Our series definition \eqref{perdatdefn} is
equivalent for the case we consider,
being a complete invariant for that class.)

\begin{definition}\label{kappadefn}
For $g\in G$, let $\kappa (g)$ denote
the conjugacy class of $g$ in $G$
($=\{g\}$ if $G$ is abelian). Let $\Conj G$ denote the free abelian
group with generators the conjugacy classes of $G$.
We also let $\kappa $ denote
the induced group homomorphism
$\ZG \to
\Conj G$ given by
$\sum n_g g \mapsto \sum n_g \kappa (g)$.
We use $\kappa$ similarly  for other induced maps.
\end{definition}

If $ (X_1,S_1,\tau_1 )$ is a $G$ extension
and $x\in \text{Fix}(S^n)$,
set $w(x)=\tau( x)\tau (Sx) \dots \tau(S^{n-1}x)$
and $\kappa_n (x)=\kappa (w(x))$.
If a topological conjugacy $\phi : X_1\to X_2$ sends $\tau_1$
to a function cohomologous to $\tau_2$,
 and $x\in \text{Fix}(S^n)$,
then $\kappa_n (x)= \kappa_n (\phi (x))$.
Given a $G$-extension of $(X,S)$ defined by $\tau $
and a conjugacy class $c$ from $G$,
define the
{\it periodic data} to be the formal
power series with coefficients in $\Conj G$,
\begin{equation} \label{perdatdefn}
  P_{\tau}= \sum_{n=1}^{\infty}
\Big(\sum_{x\in\textnormal{Fix}(S^n)}
\kappa \big(\tau( x)\tau (Sx) \dots \tau(S^{n-1}x) \big)
\Big)
 t^n \ .
\end{equation}
Then for $G$ extensions
$ (X_1,S_1,\tau_1 )$ and $ (X_2,S_2,\tau_2 )$,
a necessary and sufficient
condition for isomorphism of the
$G$ extensions obtained by restriction to
their periodic points (neglecting topology)
is that $P_{\tau_1}=P_{\tau_2}$.

\begin{definition}
Let $A$ be a square matrix over a ring.
The {\it trace series} of $A$ is
\begin{equation} \label{tsdefn}
\mathcal T_A = \sum_{n=1}^{\infty}\trace (A^n)t^n \ .
\end{equation}
For $A$ a matrix over $\ZG$, the {\it conjugacy class trace series} of $A$ is
\begin{equation} \label{ctsdefn}
\kappa \mathcal T_A =
    \sum_{n=1}^{\infty}\kappa \big(\trace (A^n)\big)t^n \ .
\end{equation}
The trace series of $A$ and $B$ are
\it{conjugate} if
 $\kappa \mathcal T_A = \kappa \mathcal T_B$.
\end{definition}
We relate $\kappa \mathcal T_A$
to existing $K$-theory invariants
\cite{SheihamCohn}
in Proposition \ref{kappatau}.
If
the extension $(X,S,\tau )$ is defined
by a matrix $A$ over $\Z_+G$, then
\begin{equation} \label{perdat}
  P_{\tau} = \mathcal T_A \ .
\end{equation}

{\bf Periodic data for G abelian. }
If $G$ is abelian,
we identify
$\kappa (g)$ with $g\in \ZG$. Then
the periodic data $P_{\tau} $
for the extension $(X,S,\tau )$
is encoded by
the usual
dynamical zeta function,
taken with coefficients in $\ZG$,
 \[
\zeta_{\tau} (z) = \text{exp}
\Big( \sum_{n=1}^{\infty}
\sum_{x: S^n x =x}
\tau (x) \tau(Sx) \cdots \tau(S^{n-1}x) \frac{z^n}n \Big) \  .
\]
When $\tau : X\to G $ is constructed from a matrix $A$ over $\Z_+G$ as above,
\begin{equation} \label{zetaequation}
\zeta_{\tau } (t)\  =\ \text{exp}\sum_{n=1}^{\infty} \frac 1n \textnormal{tr}(A^n) t^n
\ =\ (\det (I-tA))^{-1}
\end{equation}
and $\det (I-tA)$ is a complete invariant for the periodic data.
(Here, $\zeta_{\tau}$ is an example of a dynamical zeta function.
There is
a huge literature using  variants of such
functions; one survey for nonexperts
is \cite{Pollicott2011}.)

{\bf Periodic data for general $G$.}
Suppose $A$ has entries in $\Z_+ G$ where
$G$ need not be abelian. The usual polynomial
$\det(I-tA)$ need not be well
defined. Nevertheless,
by Proposition \ref{finitetracedata},
the finite sequence
$(\kappa (\trace (A^k))_{ 1\leq k \leq mn}$ determines all of $\kappa \mathcal T_A$, and the sequence
 $(\kappa (\trace (A^k))_{ 1\leq k < \infty } $ satisfies a
readily computed recursion relation with  coefficients
in $\Z$. A connection of $\kappa \mathcal T_A$ and K-theory
is described in Proposition \ref{kappatau}.

{\bf Periodic data, SE and SSE.}
If $A,B$ are SSE over $\ZG$, then
$\kappa \mathcal T_A =\kappa \mathcal T_B$
(Proposition \ref{finitetracedata}).
If $G$ is a finite abelian group, then
$\det (I-tA)$ is an invariant of SE over $\ZG$, as follows.
With $B$  SE over $\ZG$ to $A$,  by Theorem
\ref{sseclassif} there exists a nilpotent matrix $N$
such that $A\oplus N$ is SSE over $\ZG$ to $B$, and
then
\[
\det(I-tB)= \det(I-tA) \det(I-tN) = \det (I-tA)
\]
with the second equality holding by Proposition \ref{sk1fact}.

For $G$ not abelian,
$\ZG$ might contain
 nonzero nilpotent elements (for example $\mathbb{Z}[D_{4}]$, where $D_{4}$ is the dihedral group of order 4, contains nilpotent elements), and
in this case the periodic data will not be
invariant under SE over $\ZG$.
In any case, if $A$ and $B$ are SE over $\ZG$ with
lag $\ell$, then $\kappa (\trace (A^k))=
\kappa (\trace (B^k))$ for all $k\geq \ell$,
and then $\kappa \mathcal T_A=\kappa \mathcal T_B$ if and only if
$\kappa (\trace (A^k))=
\kappa (\trace (B^k))$ for all $k<\ell$.

{\bf Flow equivalence.}
Complete invariants of
$G$-equivariant  flow equivalence for $G$-SFTs
are known in terms of algebraic invariants
associated to  a presenting
$\Z_+G$ matrix $A$ (see \cite{BS05} for the case $\overline A$
primitive and \cite{bce:gfe} for the general case).


\section{Finite group extensions of SFTs via matrices
over $\ZG[t]$}
\label{sec:pmp}

Invariants of group extensions of SFTs  can be developed via
matrices over $\Z_+G$ with  the SSE/SE  approach, or via matrices
with entries from the polynomial ring $\Z_+G[t]$  with the
``positive K-theory'' approach of \cite{BW04,B02posk}).
In this section we recall and develop what we need of
the positive K-theory for constructions, and
summarize algebraic invariants in this setting.

In this paper, we formulate positive equivalence  in terms
of finite matrices.
The equivalent infinite matrix formulation
of positive equivalence described later
in this section is used
 in
\cite{BW04,B02posk}.
Other  formulations
vary a bit among \cite{BW04},
\cite{B02posk} and the present paper, but
they are equivalent where  they overlap.
The paper \cite{BW04} is written for matrices
over $\Z$ and $\Z_+$, outside of Section 7,
which address matrices over integral
group rings.

{\bf Positive equivalence.}

Let $R$ be a ring (always assumed to contain 1).
A {\it basic elementary matrix} over $R$ is a square
matrix over $R$ equal to the identity except perhaps in a single
offdiagonal entry.

Below,   $0_n$ is the $n\times n$ zero matrix, $I_n$ is
the $n\times n$ identity matrix, and $0,I$ represent zero, identity
matrices of appropriate sizes.

Let $\mathcal M$ be a set of square matrices $I-A$ over $R$
such that
\[
I-A\in \mathcal M \implies I-(A\oplus 0_n)\in \mathcal M
\ \ , \quad
\textnormal{ for all } n>0 \ \ .
\]
  Let $\mathcal S$ be a subset of $R$ containing zero and one.
A {\it basic elementary equivalence over $\mathcal S$ in $\mathcal M$}
is an equivalence of the form $ I-A \mapsto U(I-A)= I-B$
or $ I-A \mapsto (I-A)U= I-B$
 such that $U$ is a basic
elementary matrix, and both $I-A$ and $I-B$ are in
$\mathcal M$.   An equivalence $I-A \mapsto U(I-A)V=I-B$ is an
{\it elementary equivalence over $\mathcal S$ in $\mathcal M$} if
for some $k$, $(U\oplus I_k,V\oplus I_k) :
I-(A\oplus I_k) \to I-(B\oplus I_k) $ is a composition of
basic elementary equivalences over $\mathcal S$ in $\mathcal M$.
We say square matrices $I-A,I-B$ are
{\it elementary  equivalent over $\mathcal S$ in
$\mathcal M$} if there exist $j,k$ such that there is an
elementary equivalence over $\mathcal S$
in $\mathcal M$ from $I-(A\oplus I_j) $ to
$I-(B\oplus I_k)$.


\begin{definition}\label{nzcdefn} Suppose $R$ is an ordered ring with $R_+$
containing 0 and 1.
A square matrix $A$ over $R_+[t]$ has the
NZC property if for all $n\geq 0$, every diagonal entry of
$A^n$ has constant term zero.
$\nzc (R_+[t])$ is the set of square
matrices $A$ over $R_+[t] $ having the NZC property.
\end{definition}

For example,
the matrix
$
\left( \begin{smallmatrix}
t & 3+t^3 \\  2t^5 & t
\end{smallmatrix}
\right)
$
is in
$\nzc (Z_+[t]) $;  the matrix
$
\left( \begin{smallmatrix}
t & 3+t^3 \\ 1+ 2t & t
\end{smallmatrix}
\right)
$ is not.
The square matrices over $tR_+[t]$ are contained in
$\nzc (R_+[t])$.

\begin{definition}\label{defnposeq} Suppose $R$ is an ordered ring with $R_+$
containing 0 and 1. With respect to this ordered ring,
two matrices are {\it positive equivalent} if they are
elementary equivalent over $R_+$ in  $\mathcal M$,
where $\mathcal M $ is the set of square matrices
of the form $I-A$ with $A$ in $\nzc (R_+)$.

In this paper, positive equivalent without modifiers
means positive equivalent with respect to
$R=\ZG[t]$ and $R_+ = \Z_+G[t]$.
\end{definition}

{\bf Positive equivalence and SSE.}
The next result is a trivial corollary of \cite[Theorem 7.2]{BW04},
but it takes a little space to explain why this is so.
\begin{theorem}\label{sseasposeq}
  \cite[Theorem 7.2]{BW04}
Let $G$ be a group and $\ZG$
its integral group ring. Let $A,B$ be matrices in
$\nzcgpp$ and let $A^{\diamond},B^{\diamond}$ be square matrices over
$\Z_+G$ such that $I-A$ and $I-B$ are
(respectively) positive equivalent to
$I-tA^{\diamond}$ and $I-tB^{\diamond}$. Then the following are equivalent.
\begin{enumerate}
\item
$A^{\diamond}$ and $B^{\diamond}$ are SSE over $\Z_+G$.
\item
$I-A$ and $I-B$ are positive
 equivalent.
\end{enumerate}
Moreover, for every matrix $A$ in $\nzcgpp $,
there is a matrix $A^{\diamond}$ over $\Z_+G$ such that
$I-A$ is positive equivalent to
$I-tA^{\diamond}$.
\end{theorem}
\begin{proof}
  The construction in \cite[Sec. 7.2]{BW04} produces
  from $A$ in $\nzcgpp$ a matrix $A^{\sharp}$
  over $\Z_+G$ such that there is a positive equivalence from $I-A$
  to $I-tA^{\sharp}$. Then \cite[Theorem 7.2]{BW04} states (with
  different  terminology) that
  $I-A$ and $I-B$ are positive equivalent if and only if
  $A^{\sharp}$ and $B^{\sharp}$ are SSE-$\Z_+G$.
  Now assume the Claim: for any square matrix $M$ over
  $\Z_+G$,   $(tM)^{\sharp}$ is SSE-$\Z_+G$ to $M$. Then we have
  \begin{align*}
    &\ A^{\diamond} \text{ and } B^{\diamond}
    \text{ are SSE over }\Z_+G \\
    \iff &\ (tA^{\diamond})^{\sharp}\text{ and }(tB^{\diamond})^{\sharp}
    \text{ are SSE over }\Z_+G \\
\iff &\ I-tA^{\diamond}\text{ and }I-tB^{\diamond}
 \text{ are positive equivalent }\\
    \iff &\ I-A \text{ and } I-B \text{ are positive equivalent. }
    \end{align*}
  It suffices then to prove the Claim.

  Suppose $M$ is square over $\Z_+G$.
  Let $\mathcal G$ be the $G$-labeled  graph with
  adjacency matrix $M$. Let $\mathcal H $ be
  the $G$-labeled  graph with adjacency matrix $C$
  such that the  vertices of $\mathcal H$ are the   edges of $\mathcal G$,
  and $C$ is zero except that
  $C(a,b)$ is the label
  $g=g_a$ of edge $a$ in $\mathcal G$ if
  the terminal vertex of $a$ equals the inital
  vertex of $b$.
  By definition  in \cite[Sec.7]{BW04}
  (note the ``Special Case'' remark  above \cite[(2.6)]{BW04}),
  $(tM)^{\sharp}$ will be the adjacency matrix $C$ of $\mathcal H$.
  (The chosen ordering of indices to define an actual matrix
  won't affect the SSE-$\Z_+G$ class.)
  Explicitly, define matrices $R,S$, which are zero except for:
  $R(i,a)= 1$ if $i$ is the initial vertex of
  $a$; $S(a,j)= g_a$ if $j$ is the terminal
  vertex of the edge $a$. Then $M=RS$ and $C=SR$.
  \end{proof}

  \begin{notation} For a matrix $A$ in $\nzcgpp$, we will
  use $A^{\diamond}$
to denote a matrix over $\Z_+G$ such that $I-tA^{\diamond}$ is positive
equivalent to $I-A$.
\end{notation}


The connection to shifts of finite
type explained in \cite{BW04} is less straightforward for
$\nzc (\Z_+G[t])$ than for matrices over
$t\Z_+G[t]$. However,
$\nzc (R_+[t])$  is good for constructions
(e.g., it is necessary for  Proposition \ref{hms}).
  Most importantly: if in the definition
  \ref{defnposeq} of positive equivalence we replace
  $\nzc (R_+)$ with the set of square matrices over $t\Z_+G[t]$,
  then the implication $(1)\implies (2)$ of Theorem
  \ref{sseasposeq} would fail
  (see \cite[Remark 6.4]{BW04}).

The setting of positive equivalence
has been  useful for constructing
 conjugacies between SFTs and
between $G$-SFTs
  \cite{S12,S7,S6,Long2009}.
Positive equivalence constructions with matrices over
$\Z_+G$ (not over $\Z_+G[t]$) are fundamental
for the classification of $G$-SFTs up to equivariant
flow equivalence in \cite{BS05, bce:gfe}.

Recall that a matrix is {\it nondegenerate} if it has no zero row and
no zero column.  If  row $i$ or column $i$ of a matrix is zero,
then we say that the index $i$ is removable.
For a square matrix $A$, let $A=A_{0}$. Given
$A_{k}$, define $A_{k+1}=(0)$ if every index of $A_{k}$ is removable;
otherwise, define $A_{k+1}$ to be the principal submatrix of $A_k$
on the nonremovable indices. For some $k$, $A_k=A_{k+1}$, and
we call this matrix the {\it core} of $A$.
 A square matrix over $\Z_+G$ is always SSE over $\Z_+G$
to its core.

By  Theorem \ref{sseasposeq}, all matrices
  $A^{\diamond}$ over $\Z_+G$ with $I-tA^{\diamond}$ positive equivalent
  to a given $I-A$ lie in the same SSE-$\Z_+G$ class.
So, given $A$, whether the core
of $A^{\diamond}$ is  $G$-primitive does
not depend on the choice of $A^{\diamond}$.
Similarly, given $A$, the following are
equivalent:
The choices are for $A^{\diamond}$, not for the core once
  $A^{\diamond}$ has been chosen.
\begin{enumerate}
\item Some choice of
$A^{0}$ has core zero.
\item Every choice of
$A^{0}$ has core zero.
\item
Every $A^{\diamond}$ is SSE over
$\Z_+G$ to $(0)$.
\item $I-A$ is positive equivalent to
$I$.
\end{enumerate}
{\bf Some technical results.}
The main purpose of this subsection
is to prove its Propositions, which we need
later in proofs.

Suppose $A$ is a square matrix over
$t\Z_+G [t]$, say $A = \sum_{i=1}^k A_kt^k $,
with the $A_k$ matrices over $\Z_+G$.
As in \cite{BoSc3},
 define the matrix
\begin{equation}\label{sharpdefinition}
\As
=
\begin{pmatrix}
A_1 & A_2 &A_3& \dots &A_{k-2}&A_{k-1} & A_k \\
I   & 0   &0  & \dots & 0     & 0      & 0  \\
0   & I   &0  & \dots & 0     & 0      & 0  \\
0   & 0   & I & \dots & 0     & 0      & 0  \\
\dots &\dots &\dots &\dots &\dots &\dots &\dots  \\
0   & 0   & 0 & \dots &I      & 0     & 0  \\
0   & 0   & 0 & \dots &0      & I     & 0
\end{pmatrix} \ .
\end{equation}

\begin{remark} \label{sharpremark}
If $B$ is a  matrix
with all entries in $\Z_+G[t]$,
$\overline B$  is the matrix defined
by applying the
augmentation $\ZG\to  \Z$ entrywise (Definition \ref{augdefn}).
Then for $A$ over $t\Z_+G[t]$, we have
$\overline{(A^{\Box})}= (\overline{A})^{\Box}$,
and the notation $\overline \As$ is unambiguous.
\end{remark}

\begin{lemma} \label{sharplemma}
Suppose $A$ is a square matrix over
$t\Z_+G[t]$. Then the matrices $I-A$ and $I-t\As$ are
positive equivalent.
\end{lemma}
\begin{proof}

The proof is clear from the case $k=3$,
as follows. The given multiplications by elementary matrices
can be factored as a composition of basic positive equivalences.
\begin{align*}
\begin{pmatrix}
I-tA_1 & -tA_2 &-tA_3\\
-tI   & I   &0  \\
0   & -tI   &I
\end{pmatrix}
\begin{pmatrix}
I & 0 &0\\
0   & I   &0  \\
0   & tI   &I
\end{pmatrix}
\begin{pmatrix}
I & 0 &0\\
tI   & I   &0  \\
0   & 0   &I
\end{pmatrix}
\ &= \
\begin{pmatrix}
I-A & -tA_2 -t^2A_3 &-tA_3\\
0  & I   &0  \\
0   & 0   &I
\end{pmatrix}
\\
\begin{pmatrix}
I & A_2 &0\\
0   & I   &A_3  \\
0   & 0   &I
\end{pmatrix}
\begin{pmatrix}
I-A& -tA_2 &-tA_3\\
0  & I   &0  \\
0   & 0   &I
\end{pmatrix}
\ &= \
\begin{pmatrix}
I-A& 0&0\\
0  & I   &0  \\
0   & 0   &I
\end{pmatrix}
\end{align*}
\end{proof}

The next proposition is used in the
proof of Lemma \ref{zetalemma}.

\begin{proposition} \label{sizematters}
Suppose $A$ is an $n\times n$ matrix in
$\nzcgpp$
and $d$ is the maximum degree of an entry of
$A$.  Then there is a  matrix $A^{\diamond}$ over $\Z_+G$ such that
the following hold.
\begin{enumerate}
\item
$I-tA^{\diamond}$ is positive equivalent to $I-A$.
\item
$A^{\diamond}$ is $m\times m$ with $m\leq nd$.
\end{enumerate}
If $I-A$ is not positive equivalent to $I$, then
in addition $A^{\diamond}$ can be chosen to be nondegenerate.
\end{proposition}

\begin{proof}
First suppose $A\in \nzcgpp$. We claim $I-A$ is positive equivalent
to a matrix over $t\Z_+G[t]$.
This is stated for $\ZG=\Z$
in \cite[Prop. 4.3]{BW04}, but the argument is for our purposes
quite indirect, so we will sketch a proof.  Suppose for a row $i$,
the indices $j=j_1, \dots , j_t$ are those such that $A(i,j)$
has nonzero constant term, $c_{i,j} \neq 0$. For $1\leq s \leq t$,
let $E_s$
be the $n\times n$
basic elementary matrix with $E(i,j_s)=c_{i,j_s}$. Then
there is a positive equivalence from $I-A$ to
$E_1E_2\cdots E_t(I-A) := I-B_1$. $A$ and $B_1$ are equal outside row
$i$. Now, if $M_i(A)$ denotes the maximum integer $k$
such that an entry of row $i$ of $A^k$ has nonzero constant
term, then $M_i(B_1) \leq M_i(A) -1$. Thus by iterating this process,
we can produce an $n\times n$ matrix $B$ over $t\Z_+[t]$  such that
$I-B$ is positive equivalent to $I-A$.
Let $d_B$ be the maximum degree of an entry of $B$;
then
 $d_B \leq d$.

Now by Lemma \ref{sharplemma}, the matrix $I-t\Bs$ is positive equivalent to
$I-B$ and hence to $I-A$, with size $nd_B \leq nd$.
Set
 $A^{\diamond}=\Bs$.
For the nondegeneracy condition,  let
 $A^{\diamond}$ be the core of $\Bs$.
\end{proof}

The next proposition is used in the proof of
Theorem \ref{sseembed}.

\begin{proposition} \label{hms}
 Suppose $I-A,I-B$  are matrices
in $\nzcgpp$
such that $A$ and $B$ are SSE over  $\Z_+G[t]$.
Suppose $A',B'$ are matrices over $\Z_+G$
such that $I-tA'$ and $I-tB'$ are positive
 equivalent respectively to $I-A$ and $I-B$.
Then $A'$ and $B'$ are SSE over $\Z_+G$.
\end{proposition}
\begin{proof}
It suffices to prove the proposition in the case that there
are matrices $R,S$ over $\Z_+G[t]$ such that
$A=RS$ and $B=SR$.
By Theorem \ref{sseasposeq},
it suffices to show that
$I-A$ is positive  equivalent to $I-B$.
To see  this, using the ``polynomial strong shift equivalence equations''
of \cite[Sec.4]{BW04}, we
multiply by matrices
below in the order given by subscripts.
 Each multiplication gives a positive equivalence.
\begin{align*}
\begin{pmatrix}I&0\\ S&I \end{pmatrix}_4
\begin{pmatrix}I& -R\\ 0&I \end{pmatrix}_2
\begin{pmatrix}I-RS& 0\\ 0&I \end{pmatrix}
\begin{pmatrix}I&0\\- S&I \end{pmatrix}_1
\begin{pmatrix}I& R\\ 0&I \end{pmatrix}_3
&=
\begin{pmatrix}I& 0\\ 0&I-SR \end{pmatrix}
\\
\begin{pmatrix}I&B\\ 0&I \end{pmatrix}_4
\begin{pmatrix}I& 0\\ -I&I \end{pmatrix}_2
\begin{pmatrix}I& 0\\ 0&I-B \end{pmatrix}
\begin{pmatrix}I&-B\\0&I \end{pmatrix}_1
\begin{pmatrix}I& I\\ 0&I \end{pmatrix}_3
&=
\begin{pmatrix}I-B& 0\\ 0&I \end{pmatrix}
\end{align*}
\end{proof}

{\bf Infinite matrices.}

Let $R$ be a ring. $\el_n(R)$ is the group of $n\times n$ matrices
which are products of basic elementary matrices over $R$.
$\gl_n(R)$ is the group of $n\times n$ matrices invertible
over $R$.
For $R$ commutative, $\SL_n(R)$ is  the subgroup of
matrices in $\gl_n(R)$ with determinant 1.
The group $\GL(R)$ is the direct limit
group defined by the maps $\GL_n(R)\to \GL_{n+1}(R)$,
$U\mapsto U\oplus 1$.
$\El (R)$ and (for $R$ commutative) $\SL (R)$ are the subgroups of $\GL (R)$
defined as direct limits of the groups $\El_n(R)$
and $\SL_n(R)$.
We define
finite square matrices $I-A,I-B$ to be
$\El(R)$ equivalent if
if there exist $j,k,n$ and matrices $U,V$ in
$\El_n(R)$ such that $U(I-(A\oplus 0_j))V=
I-(B\oplus 0_k)$.  $\GL(R)$ equivalence and
$\SL(R)$ equivalence are defined in the same way.

For a finite square matrix $M$, let $M_{\infty}$ denote
the infinite matrix which has upper left corner $M$
and agrees with $I$ in all other entries. The elements of
$\GL(R)$ are naturally identified with the matrices
$U_{\infty} $ such that $U $ is invertible.
Similarly for $\SL (R)$ and $\El(R)$.

An equivalence $U(I-A)V$ with $U$ and $V$ in $\GL_n(R)$
produces
an equivalence $U_{\infty}(I-A)_{\infty}V_{\infty}$ by
matrices $U,V$ in $\GL(R)$. Likewise for $\El(R)$ and $\SL(R)$.
Basic elementary equivalence, ZNC and positive equivalence can
be defined for these infinite matrices in the
obvious way, such that  finite
square matrices $I-A$ and $I-B$ are positive equivalent
if and only if $(I-A)_{\infty}$ and $(I-B)_{\infty}$ are positive
equivalent.


{\bf Algebraic invariants via polynomial matrices.}

In this subsection we look at the earlier algebraic invariants
in terms of the polynomial matrix presentations.

\begin{definition} \label{elrequivdefn}
For a ring $ R$ we say square  matrices $M,N$ are
$\El (R)$ equivalent if there are positive integers $j,k,n$
and matrices $U,V$ in $\El_n(R)$ such that
$U(M\oplus I_j)V= N\oplus I_k$. $\GL (R) $ equivalence
and $\SL (R)$ equivalence are defined in the same way.
\end{definition}

Theorem \ref{sseaseq} is an easy corollary of
the main result of \cite{BoSc1}.
\begin{theorem}\cite[Corollary 6.6]{BoSc1} \label{sseaseq}
Suppose $R$ is a ring. Suppose
$I-A$ and $I-B$ are matrices over
$R[t]$;
$A',B'$ are square matrices over
$R$;  and
$I-A$ and $I-B$
are respectively
$\El(R[t])$ equivalent to
$I-tA'$ and $I-tB'$. Then
the following are equivalent.
\begin{enumerate}
\item
$A'$ and $B'$ are SSE over $R$.
\item
$I-tA$ and $I-tB$ are $\El(R[t])$ equivalent.
\end{enumerate}
\end{theorem}

If $A$ is $n\times n$ over the group ring $\ZG[t]$,
then  matrix multiplication defines
$(I-A) : (\ZG [t])^n \to  (\ZG [t])^n$ and thereby
the $\ZG[t]$ module $\cok (I-A)$.
(The isomorphism class of the module depends
in general on whether
one chooses multiplication of row vectors or column vectors.)
\begin{proposition}\cite[Theorem 5.1]{BoSc1}\label{3waystose}\footnote{See \cite{BoSc1} for attributions; especially,
$(2)\iff (3)$ is due to Fitting \cite{Fitting1936}.}
Suppose $A$ and  $B$
are square matrices over a ring $R$.
Then the following are equivalent.
\begin{enumerate}
\item
$A$ and $B$ are SE over $R$.
\item
The
 $R [t] $ modules
 cokernel $\cok (I-tA)$
and $\cok(I-tB)$ are isomorphic.
\item
$I-tA$ and $I-tB$ are $\GL (R[t])$
equivalent.
\end{enumerate}
\end{proposition}


Lastly,
we consider the algebraic invariants for the periodic data.
Proposition \ref{3waystose}
(via condition (3)) shows that
$\det (I-tA)$ is invariant under SE-$R$ for any commutative ring $R$
(e.g. $\ZG$ for $G$ abelian).
For any ring $R$,
$R[[t]]$ denotes the ring of formal power series with
coefficients in $R$, and
the {\it generalized characteristic polynomial
$\textnormal{ch}(A)$}
\cite{Pajitnov,PajitnovRanicki2000,SheihamCohn,SheihamWhitehead})
of a square matrix $A$ over $R$ is the
element of $K_1(R[[t]])$ containing  $I-tA$.
Motivation for and
a characterization of $\textnormal{ch}(A)$
are in \cite{SheihamCohn,SheihamWhitehead}.
If $R$ is commutative,
then
$\det(I-tA)$ is a complete
invariant for $\textnormal{ch}(A)$.

Recall the definitions \eqref{tsdefn} and \eqref{ctsdefn}
   for
   $\mathcal T_A$ and $\kappa \mathcal T_A$.
Given a ring $R$,
let $C$ denote
the additive subgroup
(not the ideal)
of $R$ generated by the set $\{ab-ba: a\in R, b\in R\}$.
Let $\gamma : R\to R/C$ denote the corresponding
epimorphism of additive groups.
Let $\mathcal T_A/C $ denote $\sum_{n=1}^{\infty}
\gamma(\tr A^n) t^n$. Following Sheiham
\cite[p.19]{SheihamCohn}, for a square matrix $A$ over $R$
define
$
\chi : A \mapsto \mathcal T_A/C \ .
$
\begin{proposition}\label{kappatau} Suppose $G$ is a
group and $A,B$ are square matrices over
$\ZG$. Then
\begin{equation}
\mathcal T_A/C = \mathcal T_B/C
   \  \iff
\  \kappa \mathcal T_A = \kappa \mathcal T_B \ .
\end{equation}
If $I-tA$ and $I-tB$ are
$El (\ZG [t] )$ equivalent,
or even just
$El (\ZG [[t]] )$ equivalent, then
$\kappa \mathcal T_A = \kappa \mathcal T_B$.
\end{proposition}
\begin{proof}
  The proof of the first claim is straightforward.
  For the second claim, note that
  for any ring $R$,
 $\chi$ factors through   $\textnormal{ch}(A)$,
as  pointed out by Sheiham
\cite[Remark 2.9]{SheihamCohn}.
If $I-tA$ and $I-tB$ are $\El (R[t])$ equivalent,
then  they are $\El (R[[t]])$ equivalent, so
$\textnormal{ch}(A)=\textnormal{ch}(B)$. In the
case $R=\ZG$, this means
 $\mathcal T_A/C  =\mathcal T_B/C $.
  \end{proof}
With Theorem \ref{sseaseq}, Proposition \ref{kappatau} gives an alternate
proof that $\mathcal T_A=\mathcal T_B$ when $A$ and $B$
are SSE over $R$.
For $G$ a nonabelian  group, we  do not know
if $\kappa \mathcal T_A$
determines $\textnormal{ch}(A)$.

\section{Parry's question and SE-$\ZG$} \label{sec:pqa}

\begin{pquestion} \label{billsq}
  Suppose $G$ is a finite abelian group,
  $(X,S)$ is a mixing SFT and
  $\zeta$ is a fixed dynamical zeta function.
  Must there be only finitely many topological conjugacy classes
  of $G$ extensions of $(X,S)$, with
$\zeta_{\tau}$  constructed   from a skewing function
  $\tau$ as in \eqref{zgpres}, such that
  $\zeta_{\tau}=\zeta$?
\end{pquestion}


Slightly different versions of Parry's question were recorded in
\cite[Sec. 5.3]{B02posk}, \cite[Question 31.1]{Bopen} and
\cite[Sec. 4.4, p.331]{pst2009}. The version above is matched to our
notation. The other versions are equivalent, except that
the SFT $(X,S)$ might be assumed mixing or only irreducible.
Because $(X,S)$ is fixed, a map $X\to X$ implementing an
isomorphism of $(X,S,\tau_1 )$ and $(X,S,\tau_2 )$ would have
to be an automorphism of $(X,S)$, as in the language of
\cite[Sec. 4.4]{pst2009}.
(For work on a related problem,
in which the skewing function $f$ is H\"{o}lder
into the real numbers, see \cite{PollicottWeiss2003}.)

We will address the following
version of
Question \ref{billsq}.


\begin{question} \label{billsqrevised}
 Suppose $G$ is a nontrivial finite group and
$A$ is a  \gps   matrix over
$\Z_+G$.  Let $\mathfrak M(A)$ be
the collection of \gps   matrices
$B$ over $\Z_+G$
such that
\begin{enumerate}
\item
the matrices $\overline A$ and $\overline B$
are SSE over $\Z_+$, and
\item  the matrices
$B$ and $A$ have the same periodic data,
$P_B=P_A$ as in \eqref{perdat} \\ (if $G$ is abelian,
this means
$\det (I-tB) = \det (I-tA)$).
\end{enumerate}
Must $\mathfrak M(A)$ contain only finitely
many SSE-$\Z_+G$ classes?
\end{question}

In Question \ref{billsqrevised}, the condition that
$A$ be \gps adds the requirement that the
extension be a mixing extension -- the central
case. A negative answer to (\ref{billsqrevised})
gives a negative answer to (\ref{billsq}).
The condition that $\overline A$ and $\overline B$
are SSE over $\Z_+$ captures up to isomorphism
the extensions of Question \ref{billsq} (we can
recode them to this form) and also includes
every $(X',S', \tau ' )$ such that $(X',S')$ is topologically
conjugate to $(X,S)$ and $\tau'$ gives the correct
periodic data. This does not change the
set of isomorphism classes of extensions,
because   isomorphism classes of
$G$-extensions of $(X',S')$ pull back bijectively
under topological conjugacy to
isomorphism classes of $G$-extensions of $(X,S)$.
Also,
we have broadened Parry's question to include nonabelian groups.
We  add the condition that $G$ be nontrivial for
linguistic simplicity. If $G$ is trivial, then the answer to
(\ref{billsq}) is trivially  \lq yes\rq, so we no longer need to
exclude this case when giving a negative answer.
If $G$ is nontrivial
and $A$ is \gp , then the extension
must have positive entropy, and there is nothing more
to say about excluding a case of finitely many orbits.

Parry\footnote{Descriptions of Parry's work and motivation are
based on a review of email correspondence 2002-2006
between Boyle and Parry.}
with an unpublished example showed that nonisomorphic skew
products over a mixing  SFT could share the same zeta
function $\zeta_{\tau}$. His question followed the study of dozens
of examples, and grew out a study of cocycles describing
how Markov measures change under a flow equivalence of
SFTs as in \cite{araujo}.

A natural way to attack Question \ref{billsqrevised}
is to consider how the algebraic relations
SE-$\ZG$ and SSE-$\ZG$ can refine
a prescribed $\det (I-tA)$. If the refinement is
infinite, then there is an issue of constructing
\gps  matrices realizing an
infinite class on which the algebraic
invariants differ.  In Section  \ref{parrysection},
we'll carry out this program at the level
of SSE-$\ZG$, when $\nkone(\ZG)$ is not trivial.
In this case, for  {\it every} $A$ the answer
is negative.

If $\nkone(\ZG)$ is trivial, then
SE-$\ZG$ and SSE-$\ZG$ are equivalent, by
Theorem \ref{sseclassif}.
By appeal to SE-$\ZG$  invariants,
Theorem \ref{ssezgexample} below gives a negative
answer to Parry's question for every $G$,
regardless of whether $\nkone(\ZG)$
is trivial.
However, in contrast to the SSE-$\ZG$  invariants,
the  SE-$\ZG$  invariants do not provide an infinite refinement
of the periodic data of $A$ for {\it every} $A$.
We will give  examples for
which the data $\det (I-tA)$ determines
the SE-$\ZG$ class of $A$.

$\mathfrak M(A)$ in the statement of Theorem
\ref{ssezgexample} was defined in Question \ref{billsqrevised}.

\begin{theorem} \label{ssezgexample}
Suppose  $G$ is a nontrivial finite group.  There is a \gps
matrix $A$ over $\ZG$ such that
$\mathfrak M(A)$ contains infinitely
many SE-$\ZG$ equivalence classes.
\end{theorem}


\begin{proof}
  We will define some matrices over $\ZG[t]$.
  Let $u = \sum_{g\in G}g \in \ZG$.
Fix $g$ an element of $G$ distinct from the identity $e$.
Set $s=ut\in \ZG[t]$ and $w=et$.
Below, $p_k$ in $\ZG [t]$ will depend on $k\in \Z_+$,
with $p_0=0$. Given $r$, $E_{ij}( r)$ denotes
the basic elementary  matrix of appropriate
size which equals $r$ in the $i,j$ entry and otherwise equals $I$.
Define $5\times 5$ matrices equal to $I$ except that
$U(3,4)=U(3,5)=1=V(5,1)=V(4,1)$.
$U$ will act by adding column 3 to columns 4 and 5.
$V$ will act by adding row 1 to rows 4 and 5. Define
\begin{align*}
C_k &=
\begin{pmatrix}
4s & s & s & 0 & 0 \\
4s & s & s & 0 & 0 \\
4s & 2s & 2s & 0 & 0 \\
0 & 0& 0& w & p_k \\
0 & 0& 0& 0 & w
\end{pmatrix}  \\
D_k =   U^{-1}  C_k U &=
\begin{pmatrix}
4s & s & s & s & s \\
4s & s & s & s & s \\
4s & 2s & 2s & 2s-w & 2s-w -p_k  \\
0 & 0& 0& w & p_k \\
0 & 0& 0& 0 & w
\end{pmatrix}  \\
F_k = V D_k V^{-1} &=
\begin{pmatrix}
2s & s & s & s & s \\
2s & s & s & s & s \\
2w+p_k & 2s & 2s & 2s-w & 2s-w -p_k  \\
2s -w-p_k & s& s& s+w & s+p_k \\
2s-w         & s& s& s & s+w
\end{pmatrix}  \ .
\end{align*}
We will choose $p_k$ to be a sum of $k$ monomials,
$p_k= (e-g)(t^{n_1} + \cdots + t^{n_k})$.
Define $A=F_0$. Then $A= t\As$  and $\As$ is \gp.
For each $k$, we have $\overline{p_k} =0$, and therefore
$\overline{F_k}=A$.
We will arrange the following.
\begin{enumerate}
\item
The
$\ZG [t]$ modules $\cok (I-C_k)$ are pairwise
not isomorphic.
\item
For each $k$,
there is a matrix $B_k$ over $t\Z_+G[t]$
and a finite string of matrices $F_k=B_{(0)}, B_{(1)}, \dots ,
B_{(m)}=B_k$
such that  the following hold.
\begin{enumerate}
\item
$B_k^{\Box}$
is \gps .
\item
For $1\leq i \leq m$,
$I-B_{(i)}$ equals $E_i(I-B_{(i-1)})$ or $(I-B_{(i-1)}) E_i$, for some
basic elementary  matrix
$E_i$ with offdiagonal entry in $t\ZG [t]$.
\item
For $0\leq i \leq m$,
$\overline{B_{(i)}}$ has all entries in $t\Z_+[t]$.
\end{enumerate}
\end{enumerate}
Suppose we have these conditions. For each $k$, the
$\ZG [t]$ modules $\cok (I-C_k)$,
$\cok (I-F_k)$  and, by 2(b),
$\cok (I-B_k)$ are isomorphic. Therefore the
$\ZG [t]$ modules $\cok (I-B_k)$, are, by (1), pairwise not
isomorphic. Therefore the \gps matrices $B_k^{\Box}$ are
pairwise not shift equivalent over $\ZG$.
However, the elementary equivalences of 2(b) over $\ZG[t]$
push down to elementary equivalences over $\Z[t]$, and by
2(c) these are positive equialences over $\Z[t]$. Therefore
each $\overline{B_k}$ is SSE over $\Z_+[t]$ to $A$, and
the first condition
in the Question \ref{billsqrevised} definiton of $\mathfrak M (A)$
is satisfied. For the second condition, note by 2(b) and Proposition
\ref{kappatau}
that for each $k$ the matrices $B_k^{\Box}$ and $F_k^{\Box}$ have the
same periodic data. $F_k^{\Box}$ and $C_k^{\Box}$ also have the
same periodic data. By the block structure of $C_k$, the entry
$p_k$ has no effect on the traces of powers of $C_k$. Thus
every $C_k$ has the periodic data of $C_0$, which is that of
$A$. This shows the second condition
in the Question \ref{billsqrevised} definition of $\mathfrak M (A)$
is satisfied.
So, it remains to arrange the conditions (1) and (2) above.

For condition (2), consider the multiplication of $I-F_k$ from the
right by matrices $E_{25}(s), E_{25}(s^2), \dots , E_{25}(s^{k})$,
producing say a matrix $I-G_k$. These push down to a positive
equivalence from $I-\overline{F_k}$ to $I-\overline{G_k}$.
We have
\begin{align*}
G_k(3,5)  &= (2s -w- p_k)
+ 2s^2 + 2s^3 + \cdots + 2s^{k+1} \\
G_k(4,5)  &= (s+p_k) + s^2 + s^3 + \cdots + s^{k+1} \ .
\end{align*}
 Thus for suitable $p_k$ of the specified form, these two
entries of $G_k$ will lie in $\Z_+G[t]$.
Apply the same procedure with
$E_{21} $ in place of $E_{25}$ to likewise address the sign issue
for the 1,3 and 1,4 entries. The resulting matrix is our $B_k$.

Finally, we address condition (1).
For $h$ in $G$, let $\widetilde h$ be the $|G|\times |G|$
permutation matrix which is the image of $G$ under the
left regular representation. This induces a map $M\mapsto \widetilde
M$ sending $5 \times  5$ matrices over $\ZG [t]$ to
$5|G| \times 5|G|$ matrices over $\Z[t]$.
Suppose there is an isomorphism of $\ZG [t]$ modules
$\cok  (I-C_k)  \to \cok ( I-C_j)  $.
Let the  homomorphism $\Z [t] \to \Z$ induced by $t\mapsto 1$
send a matrix $I-\widetilde C$ to $I-C'$.
Then there is an  induced  isomorphism of $\Z$ modules (abelian
groups),
$\cok (I-C'_j) \to \cok (I-C'_k)$.
The lower right $2|G|\times 2|G|$
block of $I-C'_k$ has the block form
$\left( \begin{smallmatrix} 0 & k(I-P)\\ 0&0 \end{smallmatrix} \right)$,
where $P$ is $\widetilde g$.
 From the block diagonal form of $C_j$ and
$C_k$ we conclude that $\cok (k(I-P))$ and $\cok (j(I-P))$ are
isomorphic groups.

But, let $m$ be the order of $g$ in $G$ and let $c= |G|/m$.
$P$ is conjugate by a permutation matrix to
the direct sum of $c$ copies of a matrix
$C$, where $C$  is  an $m\times m$
cyclic permutation matrix.  $I_m-C$ is
$\text{SL}_m\Z$-equivalent to $I_{m-1}\oplus 0_1$. Therefore
$\cok (k(I-P))$ is isomorphic to $(\Z/k\Z)^{(m-1)c}\oplus \Z^c$,
and for positive integers $j\neq k$, $\cok (j(I-P))$  and $\cok (k(I-P))$  cannot be
isomorphic. This contradiction finishes the proof.
\end{proof}

\begin{lemma} \label{forcedse}
Suppose $G$ is a finite group, and let
$u=\sum_g g$.
Suppose $A$ and $B$ are
matrices over $\ZG$
with some powers $A^p, B^q$ all of whose
entries lie in $u\Z$.
Suppose that $\overline A$
and $\overline B$ are SE over $\Z$.
Then $A$ and $B$ are SE over $\ZG$.
\end{lemma}

\begin{proof}
For any matrix $M$ over $\ZG$, we have
$uM=u\overline M$.  So,
 $A^{p} =
u (1/|G|) \overline{A^p} = u (1/|G|) \overline A^p$,
with $(1/|G|) \overline A^p$ having integer entries.
For $k>0$,
 $A^{p+k} =
u(1/|G|) \overline A^{p+k}$.
Without loss of generality, we suppose $p=q$.
Suppose $R,S$ gives an SE over $\Z$ of $\overline A^{\ell}$ and
$\overline B^{\ell}$:
\[
\overline A^{\ell}=RS\ , \ \overline B^{\ell}=SR \ , \
 \overline AR = R \overline B \ , \ S \overline A = \overline B S\ .
\]
Define $\widetilde R = A^pR = u (1/|G|) \overline A^p R$ and
$\widetilde S = B^pS= u (1/|G|) \overline B^p S$. Then
\begin{align*}
\widetilde R \widetilde S &=
\Big(u (\frac 1{|G|} \overline A^p R)\Big)\Big(u (\frac 1{|G|}
\overline B^p S)\Big)
=
u (\frac 1{|G|} \overline A^p R\overline B^p S) \\
&=
u (\frac 1{|G|} \overline A^p RS\overline A^p )
=
u (\frac 1{|G|} \overline A^{2p+\ell} ) = A^{2p+\ell} \ , \ \
\textnormal{ and } \\
A \widetilde R & =  A \Big(u (\frac 1{|G|} \overline A^p R)\Big)
= u \frac 1{|G|}\overline A^{p+1}  R \\
&= u \frac 1{|G|}\overline A^{p}  R \overline B
= u \frac 1{|G|}\overline A^{p}  R B=
\widetilde R B \
\end{align*}
(for the last line, note that
 $u$ lies in the center of $\ZG$).
Likewise,
$\widetilde S \widetilde R= B^{2p+\ell}$ and
$B \widetilde S
=  \widetilde S A $ .
\end{proof}


It is easy to construct matrices $A$ over $\Z_+G$
such that some power $A^p$ has all entries in $u\ZG$.
For example, take $A$ over $u\ZG$; or let $A =B + N$
where $B$ is over  $u\Z_+G$ and
$N$ over $\ZG$ is nilpotent with $uN=0$. If $B$ here is
also $G$-primitive and $B-N$ has all entries over
$\Z_+G$, then $A$ will be $G$-primitive.

\begin{lemma} Suppose $A$ is $n\times n$ over $\ZG$,
with $m=|G|$. Let $\tau_k$ denote $\trace (A^k)$,
with $\tau_{k,g}$ the integers such that $\tau_k = \sum_{g\in G}
\tau_{k,g}g$ .
Then the following are equivalent.
\begin{enumerate}
\item
There is $p$ in $\N$ such that $A^p$ has all entries in
$u\ZG$.
\item
$m\tau_{k,e} = \overline{\tau_k}$,
for $1\leq k \leq mn$,
\end{enumerate}
Now suppose a positive power of $A$ has all entries
in $u\ZG$ and $B$ is a matrix over $\ZG$ such that
(i)  $B$ and $A$ have the same periodic data or
(ii) $B$ is SE over $\ZG$ to $A$. Then
some positive power of $B$ has all entries
in $u\ZG$. Consequently, for $\mathcal R=
\Z$ or
$\mathcal R= \Z_+$:
if $\overline{A}$ and $\overline{B}$ are SE-$\mathcal R$,
then  $A$ and $B$ are SE-$\mathcal RG$.
\end{lemma}

\begin{proof}
We use $\widetilde A\colon \Z^{mn} \to \Z^{mn}$ constructed
as in Appendix \ref{zgprimse}.
Let $W$ be the subspace of $\Z^{mn}$ corresponding to
$(u\ZG)^n$. $A$ has a positive power with all entries in
 $u\ZG$ if and only if $\widetilde A$ has a power which maps
$\Z^{mn} $ into $W$ if and only if $\widetilde A$ restricted to
the complementary invariant subspace is nilpotent. This
holds if and only if the sequences
$(\trace (\widetilde A^k))_{1\leq k \leq mn}$ and
$(\trace ((\widetilde A|_W)^k))_{1\leq k \leq mn}$ are
equal. We have $\trace (\widetilde A^k)
=m\tau_{k,e}$ and  (because $A$ acts on $(u\ZG^n)$ exactly
as $\overline A $ acts on $\Z^n$)
$\trace ((\widetilde A|_W)^k)= \overline{\tau_k}$.
This proves the equivalence of (1) and (2).

Then (i) holds because (1)$\iff$(2) shows (1)
depends only on the periodic data. Although the
periodic data need  not be an invariant of SE-$\ZG$
when $G$ is nonabelian, if matrices $A,B$ are
SE-$\ZG$ then for every large enough $\ell \in \N$
there are $R,S$
over $\ZG$ such that  $A^{\ell} =RS$ and $B^{\ell} =SR$,
and then  $A^{2\ell} = (A^{\ell}R)S$ and $B^{2\ell} =S(A^{\ell}R)$.
Clearly if $A^{\ell}$ is over $u\ZG$, then so is $B^{2\ell}$.
The final claim follows now from Lemma \ref{forcedse}.
\end{proof}

\begin{proposition} \label{cayley}
Suppose $G$ is a finite abelian group. Set
$u=\sum_g g$. Let $\mathcal Z_u$ denote the
set of polynomials of the form
$1 + \sum_{i=1}^kc_iut^i$, with each $c_i$ in $\Z$.
Suppose $A$ and $B$ are square matrices over
$\ZG$ such that
$ \det(I-tA)$ and $\det (I-tB)$ lie in $\mathcal Z_u$,
and $\overline A$ and $\overline B$ are SE over
$\Z$.
Then $A$ and $B$ are SE over $\ZG$.
\end{proposition}

\begin{proof}
By the Cayley-Hamilton Theorem, for all large $n$
the matrices $A^n$ and $B^n$ have entries in $u\Z$.
The theorem then follows from Lemma \ref{forcedse}.
\end{proof}

Proposition \ref{cayley} applies to any $A$ all of whose entries
are integer multiples of $u$; for example,
$A=(e+g)$ with $G =\{e,g\} = \Z/2\Z  $.

In the case of $A$ satisfying the assumptions
of Proposition \ref{cayley}, with $\nkone (\ZG)$ trivial,
 to answer Parry's question
we are left with the open problem: for $A$ \gp,
can the refinement the  SSE-$\ZG$ class of $A$
by SSE-$\ZG_+$ be infinite?
In the case $G=\{e\}$ ($\ZG = \Z$),
that question remains open more than 40 years
after Williams' original paper \cite{Williams73}.

\begin{remark} \label{ParryLivsicRemark}
In \cite[Theorem 7.1]{Parrylivsic}, Parry proved that for $G$ compact and $X$ an irreducible SFT if $f,g:X \to G$ are H\"{o}lder with equal weights on all periodic
points, then $f$ and $g$ are H\"{o}lder cohomologous. If one assumes
only that the $f$ and $g$ weights are conjugate, Parry shows then the
existence of an isometric automorphism $\phi$ of $G$ such that $\phi
f$ and $g$ are cohomologous \cite[Theorem 6.5]{Parrylivsic}.
Parry also gives an example with $G$ finite of two cocycles having conjugate weights for which the isomorphism is necessary, although the example is not mixing \cite[Section 10]{Parrylivsic}.

We note now that in general $\phi$ cannot be chosen
to be the identity even if the extension is mixing (i.e.,
presented by a $G$-primitive matrix $A$ over $\Z_+G$).
 For example, let $G$ be a finite group having an outer
automorphism $\varphi$ for which $\varphi$ preserves all conjugacy
classes of $G$. Such groups exist (see \cite{BrooksbankMizuhara14}); for
example, the group $LP(1,\mathbb{Z}/8)$ consisting of all
linear
permutations
 $x \mapsto \sigma x +\tau $ on $\mathbb{Z}/8$, with $\sigma, \tau$
in $\Z /8$, is such a group. Let $A$ be primitive over
$\mathbb{Z}_{+}G$, and $\tau$ denote the corresponding edge labeling
on the graph of $\overline{A}$ coming from $A$. Then $\phi \tau$ is
another edge labeling, and
$\phi \tau$ and $\tau$ have
conjugate weights on all periodic points.
However, $\phi \tau$ and $\tau$ are not cohomologous.  If they
were, then because they are defined by edge labelings of an
irreducible graph,
by \cite[Lemma 9.1]{Parrylivsic}
there would be a function $\gamma :X_A\to G$
such that $\gamma (x)$ depends only on the initial vertex of $x$ and
$\phi \tau = \gamma^{-1} \tau \gamma$ .
Let $\nu$ be a vertex and let $g$ in $G$ be such that
$g = \gamma(x)$ when $x_0$ has initial
vertex $\nu$. Now for every word $x_0\dots x_k$ beginning and
ending at $\nu$:  if $h= \tau (x_0)\cdots\tau(x_k)$,  then
$\phi (h) = g^{-1}hg$. Because $A$ is $G$-primitive, every element of
$G$ occurs as such an $h$, and therefore $\phi$ is an inner
automorphism. This contradiction shows $\phi \tau$ and $\tau$ are
not cohomologous.
 \end{remark}


\section{Parry's question and SSE-$\ZG$}\label{parrysection}

In this section, we prove the following result,
which gives a strong negative  answer to
Parry's question (\ref{billsq}) whenever
$\textnormal{NK}_1(\ZG)\neq 0$.
(See Appendix \ref{sec:nk1} for a description of the
finite $G$ with nontrivial $\textnormal{NK}_1(\ZG)$.)


\begin{theorem}\label{parryanswer}
Let $G$ be a finite group such that
$\textnormal{NK}_1(\ZG)\neq 0$.
Let $(X,\sigma)$ be a mixing shift of finite type and let
$\tau: X\to G$ be a continuous function defining a mixing
$G$-extension $(X_{\tau},\sigma_{\tau})$  of $(X,T)$.

Then there is an infinite family
of $G$-extensions of $(X,T)$
which are eventually conjugate as $G$-extensions
to $(X_{\tau},\sigma_{\tau})$
and which are pairwise not isomorphic $G$-extensions.
If $G$ is abelian, then they
all have the same dynamical zeta function.
\end{theorem}

Theorem  \ref{parryanswer} will be proved as a
corollary to the following result.

\begin{theorem}
\label{zginfinite}
Suppose $G$ is a finite  group and $A$ is a \gps
matrix with spectral radius $\lambda>1$ and $\nkone(\ZG) \neq 0$.
Let $A$ be a \gps  matrix.

Then there is an infinite family $\{A_{i}: i\in \mathbb N\}$
of \gps  matrices
which are pairwise not SSE over $\ZG$ but such that for all $i$
the following hold:
\begin{enumerate}
\item
$A_i$ is SE over $\Z_+G$ to $A$.
\item
$\overline{A_i}$ is SSE over $\Z_+$ to $\overline A$.
\item
If $G$ is abelian, then
$\det(I-tA_{i}) = \det(I-tA) $ .
\end{enumerate}
\end{theorem}

To prove Theorem \ref{zginfinite},
we first will work to establish a
rather technical result,
Proposition \ref{sseembed}.
Below,  we will use the
notations of
(\ref{zgtentrynotation}) and the definitions
\eqref{sharpdefinition}, \eqref{defngprimitive}
and \eqref{specradiusdefn}
of a matrix $\As$,
a   $G$-primitive matrix
and  the spectral radius
$\lambda_A$ of a square matrix over
$\ZG$ or $\ZG[t]$. For a polynomial $p$
over $\ZG$, $\lambda_p$ is the spectral radius of the $1\times 1$
matrix $(p)$. For a polynomial matrix $M=M(t)$,
we let $M(1)$ denote its
evaluation at $t=1$.
\begin{lemma} \label{zetalemma}
  Suppose $n>1$ and  $A$ is an $n\times n $   matrix
over $t\Z_+G[t]$
    with spectral radius
$\lambda >1$ and with $\As$   \gps.
Given $\epsilon >0$, there exists a positive integer
$m_0$ such that for any $d \geq m_0$
there is an $n\times n$ matrix $C$ over
$t\Z_+G[t]$ such that $I-C$
is positive equivalent
to $I-tA$ and
\[
c_{11kg} > (\lambda -\epsilon)^k\ , \ \
\text {for } m_0\leq k \leq d\ ,\
\text{ for all } g \in G \ .
\]
\end{lemma}
\begin{proof}
We will produce $C$ in three stages.

STAGE 1. Because $A(1)$ is $G$-primitive,
by \cite[Lemma 6.6]{BS05} there is a positive equivalence with respect to the ordered ring $(\ZG,\Z_+G)$ from $I- A(1)$ to a matrix $I-H$ such that $H$ is a
matrix over $\Z_+G$ with no zero entry. Lift this positive equivalence with respect to $(\ZG,\Z_+G)$ to a positive equivalence with respect to $(\ZG[t],\Z_+G[t])$
from $I-tA$ to a matrix $I-L$, with $L$ a matrix over $t\Z_+G[t]$ with every entry nonzero.

STAGE 2.
In this stage, given $\epsilon >0$ we produce an
$n\times n$  matrix $B$ over $t\Z_+G[t]$
with no zero entry such that $I-tA$
is $\ZG [t] $ positive equivalent to $I-B$
and  the $B(n,n)$ entry has spectral radius
greater than $\lambda -\epsilon /2$.

 For this, we define $n\times n$ matrices $B_1, B_2, \dots $ recursively.
We
set $B_1$ to be the matrix $L$ produced in Stage 1.
In block form, let $B_1=
\left( \begin{smallmatrix}
M & u \\
v & f
\end{smallmatrix}
\right)$, in which $f$ is $1\times 1$.
A matrix $B_k$ will have  a block form
\begin{equation} \label{blockform}
B_k \ =\
\begin{pmatrix}
M & u \\
v^{(k)} & f^{(k)}
\end{pmatrix} \ .
\end{equation}
Given $B_k$,
define $B_{k+1}$ by the equivalence
\[
I-B_{k+1} \ =\
\begin{pmatrix}
I & 0 \\
v^{(k)} & I
\end{pmatrix}
\begin{pmatrix}
I-M & -u \\
-v^{(k)} & 1- f^{(k)}
\end{pmatrix}
\ = \
\begin{pmatrix}
I-M & -u \\
-v^{(k)}M & 1- f^{(k)} -v^{(k)}u
\end{pmatrix} \ .
\]
This defines a positive equivalence
from $I-B_k$ to $I-B_{k+1}$.
 By induction, for all $k$, $B_k$ is in $\mathcal M$
and has no zero entry; $B_k^{\Box}$ is  \gps; and
$B_{k+1}$ is a matrix over
$t\Z_+G [t]$ with block form
\[
B_{k+1} \ =\
\begin{pmatrix}
M & u \\
vM^k & f+ v(I+M+\cdots + M^{k-1})u
\end{pmatrix} \ .
\]

Because $A$ is \gps , by the condition (3) in Theorem \ref{zgperron} we have a positive real number $c$  such that $\trace (A^j) > c\lambda^j (g_1+\cdots + g_m)$
for all large $j$.
Because $M^{\Box}$ is a proper principal submatrix of the
\gps  matrix $(B_1)^{\Box}$, which has spectral
radius $\lambda$,  we have
$\lambda_M < \lambda$.
Choose $\delta >0$ such that $\delta < \epsilon /2$ and
$\lambda_M <\lambda - \delta $.
For all large $j$,
\[
\trace(M^j ) < (\lambda -\delta)^j(g_1+\dots +g_m) \ .
\]
Because $M^k$ has entries in $t^k\ZG[t]$
and $u$ has entries in $t\ZG [t]$,  if
$j\leq k$ then
\[
\trace\big((A^{\Box})^j \big)
\ = \
\trace\big((M^{\Box})^j\big)
+\trace\big((f^{(k)})^{\Box})^j\big)
\ .
\]
It follows that for all large $k$, for $j\in \{k-1,k\}$,
\begin{equation}\label{sharpiebound}
\trace\big(((f^{(k)})^{\Box})^j\big)
\geq (\lambda - \delta )^j(g_1+ \cdots + g_m) \ .
\end{equation}
Consequently, $f^{(k)}$ is \gps
for all large $k$.

Let $\lambda^{(k)}$ be the spectral radius of
$(f^{(k)})$.
Let $d$ be the maximum degree of an
entry of $B_1$.  From the block form
\eqref{blockform} we see that $f^{(k)}$
has degree at most $dk$. Then
by Proposition \ref{sizematters}, we can use a
version of
$p^{\Box}$ which is a
$dk \times dk$ matrix $Q$ over $\Z_+G$.
Then
for $q=dk/m$,
the matrix $\overline Q$ is $q\times q$
over $\Z_+$ with spectral radius
$\lambda_Q=\lambda^{(k)}$.
Using \eqref{sharpiebound}, we have
\[
\lambda^{(k)}=\lambda_Q
\geq
\Big( \frac 1q \trace (Q^k)
\Big) ^{1/k}
\geq
\Big( \frac {m}{dk}
(\lambda-\delta)^km)
\Big) ^{1/k} \ .
\]
Because $0 <\delta < \epsilon /2$,
It follows that
$\lambda^{(k)} >  \lambda -\epsilon /2  $
for all
 large $k$.

STAGE 3.
We define $n\times n$ matrices $P_1, P_2, \dots $
over $t\Z_+[t]$
recursively. The recursive step is the same as in Stage 2, but
with row 1 in Stage 3 playing the role of row $n$ in Stage 2.
In block form, we write
$P_1=\left( \begin{smallmatrix}
s & w \\
x & Q
\end{smallmatrix} \right)
$, with $s$ being $1\times 1$.
We take
$P_1=B$ from  Stage 2 and set
$q=P_1(n,n)$. The $1\times 1$ matrix $\begin{pmatrix} q\end{pmatrix}$ has $q^{\Box}$  \gps  with spectral radius $\lambda_q$ such that
$0< \lambda - \lambda_q < \epsilon /2$ .

A matrix $P_k$ will have  a block form
\[
P_k\ =\
\begin{pmatrix}
s^{(k)} & w^{(k)} \\
x & Q
\end{pmatrix}
\]
and given $P_k$ we define $P_{k+1}$  by
\[
I-P_{k+1} \ =\
\begin{pmatrix}
1 & w^{(k)} \\
0 & I
\end{pmatrix}
\begin{pmatrix}
1-s^{(k)} & -w^{(k)} \\
-x & I-Q
\end{pmatrix}
=
\begin{pmatrix}
1-s^{(k)}-w^{(k)}x & -w^{(k)}Q \\
-x & I-Q
\end{pmatrix} \ .
\]
By induction,
\[
P_{k+1} \ =\
\begin{pmatrix}
s+w(I+Q+\cdots +Q^{k-1})x & wQ^k \\
x & Q
\end{pmatrix}
\]
and  $q=P_{k+1}(n,n)$.
As in Proposition \ref{looprate}  , let $(\tau_j)$ be the sequence from $\Z_+G$
such that
\[
 \sum_{k=1}^{\infty} q^k = \ \sum_{j=1}^{\infty} \tau_j t^j \ .
\]
Appealing to
Proposition \ref{looprate}, choose positive $c',d'$ such that
$\tau_j > c'(\lambda_q)^j (g_1+\cdots +g_m)$ for all $j\geq d'$.
Pick  $g,h$ in $G$ and
positive integers $n_1,n_2$ satisfying
$gt^{n_1}\leq P_1(1,n)$ and
$ht^{n_2}\leq P_1(n,1)$.
Then for $k> d'$,
\begin{align*}
P_{k+1}(1,1) \ &\geq\ P_1(1,n)\ (Q(n,n))^{k-1}\ P_1(n,1) \\
&\geq \ gt^{n_1}\ \Big(
\sum_{j=d'}^{k-1}c'(\lambda_q)^j(g_1+\cdots + g_m)t^j\Big)\ ht^{n_2} \\
&\geq\
\sum_{j=d'+n_1+n_2}^{k+n_1+n_2-1}
\Big(\frac{c'}{(\lambda_q)^{n_1+n_2}}\Big)
(\lambda_q)^j (g_1+\cdots + g_m)t^j \\
&> \
\sum_{j=d'+n_1+n_2}^{k+n_1+n_2-1}
\Big(\frac{c'}{(\lambda_q)^{n_1+n_2}}\Big)
\Big(\lambda - \frac{\epsilon}2\Big)^j (g_1+\cdots + g_m)t^j
\ .
\end{align*}
Let $m_0$ be the smallest $j$ such that
$j\geq d'+n_1+n_2$ and
\[
\Big(\frac{c'}{(\lambda_q)^{n_1+n_2}}\Big)
\Big(\lambda - \frac{\epsilon}2\Big)^j
\ >\ (\lambda -\epsilon)^j \ .
\]
Then given $d\geq m_0$, for $k=d$ we have
$P_k(1,1) > \sum_{j=m_0}^d (\lambda -\epsilon)^j (g_1+\cdots + g_m)t^j$ .
This finishes the proof of the lemma.
\end{proof}

\begin{proposition} \label{sseembed}
  Suppose $A$ is an $n\times n$ \gps  matrix over $\Z_+G$,  $n>1$
and $1< \beta < \lambda_A$. Then there is a positive integer $r_0$ such that the following holds. If $r\geq r_0$ and
 $I-Q$ is a matrix
in $\gl (k, \Z [t])$ such that
\begin{enumeratei}
\item
$|q_{ijsg}| \leq \beta^s$ for all $i,j,s,g$, and
\item
$Q\in \mathcal M(t^r\Z[t])$
\end{enumeratei}
then the matrix
$\begin{pmatrix} I-Q & 0 \\ 0 & I-tA
\end{pmatrix} $
is $\el (\Z G[t])$ equivalent to an
$(m+k)\times (m+k)$ matrix $I-B$ over $\Z G[t]$
such that

\begin{enumerate}
\item
$B$ has entries in $t\Z_+ G[t]$
\item
$B^{\Box}$ is \gps
\item
if
$\overline Q=0$, then
$\overline B^{\Box} $ is SSE over $\Z_+$ to $\overline A$.
\end{enumerate}
\end{proposition}

\begin{proof}
We use $\sim$ to denote $\EL (\ZG [t])$ equivalence.
First, note that if $I-F$ is a matrix over $\Z G[t]$
with block form
$I-F=
\begin{pmatrix} I-Q & -X \\ 0 & I-C
\end{pmatrix}
$
such that $I-C \sim I-tA$, then
the invertibility of $I-Q$ implies
$I-F \sim
\begin{pmatrix} I-Q & 0 \\ 0 & I-tA
\end{pmatrix} = I-(Q\oplus tA)$, since
\begin{align*}
\begin{pmatrix} I-Q & -X & 0\\ 0 & I-C & 0 \\ 0 &0 & I
\end{pmatrix}
&
\begin{pmatrix} (I-Q)^{-1} & 0 & 0\\ 0 & I & 0 \\ 0 &0 & (I-Q)
\end{pmatrix}
\begin{pmatrix} I & X & 0\\ 0 & I & 0 \\ 0 &0 & I
\end{pmatrix}
\begin{pmatrix} I-Q & 0 & 0\\ 0 & I & 0 \\ 0 &0 & (I-Q)^{-1}
\end{pmatrix} \\
= \ &
\begin{pmatrix} I-Q & 0 & 0\\ 0 & I-C & 0 \\ 0 &0 & I
\end{pmatrix} \ .
\end{align*}
Next, given $\beta$, let $\epsilon = (\lambda_{A} - \beta)/2$ and
let $m_o$ be the integer of the conclusion of Lemma
\ref{zetalemma} given  $A$ and $\epsilon$.
Suppose $I-Q\in \gl (k,\Z[t])$ and $Q$ satisfies
(i) and (ii).
Pick $r\geq m_0$ such that for all
$s\geq r$, $(\lambda_{A} - \epsilon)^s > 2k\lceil \beta^s\rceil +1$ .
Let $d$
be an integer such that $d>r$ and $d\geq \text{degree}(Q)$.
Now take $I-C$ from Lemma \ref{zetalemma}, positive equivalent
to $I-tA$, such that
\[
c_{11gm} \geq  2k \lceil \beta^m \rceil +1\ , \
\text{ for } r\leq m\leq d\  \text{ and for all } g \ .
\]
Let $u= \sum_g g$. Let $\alpha =
\sum_{m=r}^d \lceil \beta^m\rceil u t^m$ .
Consider a matrix in block form,
\[
H=
\begin{pmatrix} Q & X\\ 0 & C
\end{pmatrix} \
= \
\left(
\begin{array}{cccc|cccc}
q_{11} &q_{12} & \cdots &q_{1k}& \alpha & 0  & \cdots & 0 \\
q_{21} &q_{22} & \cdots &q_{2k}& \alpha & 0  & \cdots & 0 \\
\vdots &      &       &\vdots & \vdots &    &        &\vdots  \\
q_{k1} & q_{k2} & \cdots &q_{kk}& \alpha & 0  & \cdots & 0 \\ \hline
0  & 0  & \cdots  & 0        & c_{11}  & c_{12}&\cdots & c_{1n}  \\
0  & 0  & \cdots  & 0        & c_{21}  & c_{22}&\cdots & c_{2n}  \\
\vdots &&        &\vdots   &\vdots    &       &       &\vdots \\
0  & 0  & \cdots  & 0        & c_{n1}  & c_{n2}&\cdots & c_{nn}
\end{array}
\right)
\  .
\]
Define a matrix $V$ with matching block structure,
$
V=
\begin{pmatrix}
I_k  & 0 \\
Y  & I_{n}
\end{pmatrix}
$,
in which the top row of $Y$ has every entry 1 and the other
entries of $Y$ are zero, and
in which
$I_j$ as usual  denotes a $j\times j$ identity matrix.
Define $B=V^{-1}HV $. We have
\begin{equation} \label{hsim}
B =
\left(
\begin{array}{cccc|cccc}
q_{11}+\alpha &q_{12}+\alpha  & \cdots &q_{1k}+\alpha & \alpha & 0  & \cdots & 0 \\
q_{21}+\alpha  &q_{22}+\alpha & \cdots &q_{2k}+\alpha & \alpha & 0  & \cdots & 0 \\
\vdots &      &       &\vdots & \vdots &    &        &\vdots  \\
q_{k1}+\alpha  & q_{k2}+\alpha  & \cdots &q_{kk}+\alpha & \alpha & 0  & \cdots & 0 \\ \hline
x-\eta_1  & x-\eta_2  & \cdots  & x-\eta_k         & x  & c_{12}&\cdots & c_{1n}  \\

c_{21}  & c_{21}  & \cdots  & c_{21}
& c_{21}  & c_{22}&\cdots & c_{2n}  \\
\vdots &&        &\vdots   &\vdots    &       &       &\vdots \\
c_{n1}  & c_{n1}  & \cdots  & c_{n1}
& c_{n1}  & c_{n2}&\cdots & c_{nn}
\end{array}
\right)\
\end{equation}
in which $x=c_{11}-k\alpha$
and $\eta_j=q_{1j}+q_{2j}+\cdots + q_{kj}$. Then
$x \geq (k+1)u(t^d + \cdots + t^r)$,
$x-\eta_j \geq u(t^d + \cdots + t^r)$ and
$q_{ij}+\alpha \geq 0$. Because $x$ is \gps  and
$C$ is \gpscomma\    it follows easily that $B$ is $G$-primitive.
 Also, since
$I-B=V^{-1}(I-H)V$, the matrix $I-B$ is $\el (\Z G[t])$ equivalent
to $I-H$, and therefore to $I-t(Q\oplus C)$.

Finally, suppose $\overline Q=0$.
We must show
$\overline{B^{\Box}}$
is SSE over $\Z_+$ to $ A$.
Clearly $A$ and $\overline{C^{\Box}}$
are SSE over $\Z_+$.
The matrices $B$ and $C$ have all entries
  in $t\Z_+[t]$. Thus by
  Remark \ref{sharpremark},
  $\overline{B^{\Box}}=
\overline{B^{\Box}}$
and
$\overline{C^{\Box}}
=\overline{C}^{\Box}$.
Therefore  it suffices to show that $\overline{B}^{\Box}$ and
$\overline{C}^{\Box}$
are SSE over $\Z_+$.
By Proposition \ref{hms}, this will follow if we show
$\overline B $ is SSE over $\Z_+[t]$ to $\overline C$.

Because $\overline Q=0$,
we have $\overline{B}=\overline{H'}$,
 where $H'$ is the matrix obtained from $H$ by replacing
the entries $q_{ij}$ and $\eta_j$
in the display \eqref{hsim}
with zero. Let $D$ be the lower right hand block of the
$2\times 2$ block matrix $B$.
$H'$ is SSE over $\Z_+G[t]$ to
$C$, since
\[
C=
\begin{pmatrix} Y & I_n
\end{pmatrix} \
\begin{pmatrix}  X\\  D
\end{pmatrix} \
\quad \text{and} \quad
H'=
\begin{pmatrix}  X\\  D
\end{pmatrix} \
\begin{pmatrix} Y & I_n
\end{pmatrix} \ .
\]
Therefore
$\overline{B}=\overline{H'}$ is SSE over $\Z_+[t]$ to
$\overline C$. This finishes the proof.
\end{proof}

\begin{lemma} \label{amalglemma}
Suppose $G$ is a finite group, $N$ is nilpotent $n\times n$
over $\ZG$ and $r\in \N$.  Then
there is a matrix $M_r$ over $t^r\ZG [t]$
such that $\overline M_r  =0$ and
$I-M_r$
is $\el (n,\ZG[t])$-equivalent  to $I-t^rN$.
Given $N$, the matrices $M_r$ can be chosen
such that the coefficients of all entries
are bounded above independent of  $r$.
\end{lemma}
\begin{proof}

Suppose $N$ is $n\times n$. Because $\overline N$ is nilpotent over $\Z$,
we can take $U$ in $\SL_n(\Z )=\EL_n(\Z )$ such that the matrix
$N_1=U^{-1}\overline N U $ is  upper triangular with zero diagonal.
Given $r$, for $1\leq i<n$, let $W$ be $n\times n$ with
$W(i,j) = -t^rN_1(i,j)$ if $i<j$ and $W=I$ otherwise.
Set $W= W_1W_2\cdots W_{n-1}$; then
$W \in \EL (n,\ZG [t])$ and
$\overline{W(I-t^rN_1)}=I$.
Let $M_r$ be the matrix over $t\ZG [t]$
such that $I-M_r = WU^{-1}(I-t^rN)U
=W(I-t^rN_1)$.
Then $I-\overline M_r= \overline{I-M_r}
=\overline{W(I-t^rN_1)} =I$, so $\overline M_r = 0$.
The boundedness claim is clear from the construction.
\end{proof}

\begin{lemma}
\label{suspend}
Suppose $G$ is a finite group and $A$ is a \gps
matrix with spectral radius $\lambda>1$ and
$N$ is nilpotent over $\ZG$. Then for all sufficently
large $r$ in $\N$, the matrix
$\begin{pmatrix}
I-tA & 0 \\
0 & I-t^{r} N
\end{pmatrix}$
is $\el (\ZG[t])$-equivalent to a
matrix $I-B$ such that
$B$ has entries in $t\ZG_+[t]$
and $B^{\Box}$ is \gps  and
$\overline{B^{\Box}}$ is SSE over $\Z_+$ to $\overline A$.
\end{lemma}
\begin{proof}
Pick $\beta$ such that
$1< \beta < \lambda$. Let $r_0$ be the integer of
Proposition \ref{sseembed}, which depends on
$A$ and $\beta$. Let $\{M_r\}$ be the uniformly bounded family
given for $\{t^rN\}$ by Lemma
\ref{amalglemma}.
Then for all large $r\in \mathbb N$, $r\geq r_0$ and the matrix $Q= M_r$
satisfies $|q_{ijs}|\leq \beta^s$ for all $i,j,g,s$. Because $t^rN$ is nilpotent, the
matrix $I-t^rN$ is invertible over $\ZG [t]$. Now Lemma \ref{suspend} follows from Proposition \ref{sseembed}.
\end{proof}
Given $r \in \mathbb{N}$, define $V_{r}:NK_{1}(\mathbb{Z}G) \to NK_{1}(\mathbb{Z}G)$ by $V_{r}:[I-tN] \mapsto [I-t^{r}N]$, and $F_{r}:NK_{1}(\mathbb{Z}G) \to NK_{1}(\mathbb{Z}G)$ by $F_{r}:[I-tN] \mapsto[I-tN^{r}]$. The map $V_{r}$ is often called the Verschiebung operator, and $F_{r}$ the Frobenius operator.
\begin{lemma}
\label{VFoperations}
Let $G$ be a finite group and $r \in \mathbb{N}$ be such that $r$ and $|G|$ are relatively prime. Then the map $V_{r}:NK_{1}(\mathbb{Z}G) \to NK_{1}(\mathbb{Z}G)$ is injective.
\begin{proof}
One may check directly that $F_{r}V_{r}(x) = rx$ for all $x \in NK_{1}(\mathbb{Z}G)$. By a result of Weibel \cite[6.5, p. 490]{WeibelMayer}, the order of every element in $NK_{1}(\mathbb{Z}G)$ must be a power of $|G|$. Thus the map $F_{r}V_{r}$ is injective for $r$ relatively prime to $|G|$, and $V_{r}$ is as well.
\end{proof}
\end{lemma}

\begin{proof}[Proof of Theorem \ref{zginfinite}]
Because
$\textnormal{NK}_1(\ZG) $ is nontrivial,
it is infinite \cite{Farrell1977}.
Given $j\in \mathbb N$,
let $N_1, \dots , N_j$ be nilpotent
over $\ZG$ with
the matrices $I-tN_j$ representing distinct classes of $\nkone (R)$.
For a sufficiently large such $r$, Lemma \ref{suspend} applies to each $t^{r} N_i$, giving $B_i$ satisfying the conclusions of the lemma.
We take $r$ which in addition is relatively prime to $|G|$;
then the matrices $I-t^{r}N_i$ will represent distinct classes of $\nkone (\ZG)$,
by Lemma \ref{VFoperations}. Let $A_i=B_i^{\Box}$.
Condition (2) holds as part of Lemma \ref{suspend}.
Condition (1) holds because (i)
adding a nilpotent direct summand to a matrix does not affect
its SE class and (ii) for $G$-primitive matrices, SE over $\ZG$
is equivalent to SE over $\Z_+G$ (Prop. \ref{primitiveeventual}).

By Theorem
\ref{sseclassif},
the matrices  $A_i$ are pairwise
not SSE over $\ZG$. Condition (3) holds because
$\det (I-tA_i)=\det(I-tA) \det(I-tN_i)$ and
$\det(I-tN_i)$ here must be 1 by Prop. \ref{sk1fact}.
\end{proof}


\begin{proof}[Proof of Theorem \ref{parryanswer}]


Let $A$ be a \gps  matrix  defining a $G$ extension
which is isomorphic to that defined by $\tau$ and let $A_{i}$
be the \gps  matrices provided by Theorem \ref{zginfinite}.
By condition (1) of Theorem \ref{zginfinite} and Proposition
\ref{eventualconjugacy},
these $G$ extensions of $(X,T)$ are
all eventually conjugate  to $(X_{\tau},\sigma_{\tau})$.
By condition (2), the $A_{i}$ define $G$ extensions which are
conjugate to $G$-extensions defined from $(X,T)$.
Because the $A_{i}$ are not SSE over $\ZG$, they cannot be
SSE over $\Z_+G$, so their extensions (and hence their conjugate
extensions from $(X,T)$) are pairwise not isomorphic.
Lastly, they satisfy condition (3), which for abelian $G$
is a well defined  invariant of
SSE over $\ZG$ (and even SE over $\ZG$) and therefore is
carried over to the isomorphic versions defined over
$(X,T)$.

\end{proof}

\section{Open problems} \label{sec:open}

\begin{realizationproblems}
This set of problems for the algebraic analysis of
mixing finite group extensions of SFTs involves understanding
the range of the algebraic invariants.
\begin{enumerate}
\item
Suppose $G$ is  finite group, $A$ is \gps and $N$
is a nilpotent matrix over $\ZG$. Must $A\oplus N$ be
SSE over $\ZG$ to  a \gps  matrix?
\\ (The methods for
Section \ref{parrysection} and
\cite[Radius Theorem]{BH91} might be useful.
The answer to the corresponding problem for
matrices over subrings of $\R$ is positive
\cite{BoSc3}.)

\item Given a finite abelian group $G$,
characterize the polynomials $\det (I-tA)$
arising from  \gps matrices $A$ over $\ZG$. \\
 (For $\ZG = \Z \{e\} =\Z$, this is solved \cite{S8}.)

\item Given a finite  group $G$,
characterize the trace series
$\mathcal T_A$ and conjugate
trace series
$\kappa\mathcal T_A$
arising from  \gps matrices $A$ over $\ZG$.

\item
Let  $G$ be a finite abelian group.
Suppose  $A$ is a \gps matrix over $\Z_+G$,
and $B$ is a matrix over $\ZG$ such that
$\det (I-tA) = \det (I-tB)$.
Must
$B$ be shift equivalent over $\ZG$ to  a
\gps matrix? \\
 (For  analogues involving
$\R$ and $\Z$, see \cite{BoSc3}.)

\item
Let  $G$ be a finite  group.
Suppose  $A$ is a \gps matrix over $\Z_+G$,
and $B$ is a matrix over $\ZG$ with the same
conjugate trace series \eqref{ctsdefn},
$\kappa \mathcal T_A = \kappa \mathcal T_B$.
Must
$B$ be shift equivalent over $\ZG$ to  a
\gps matrix?
\end{enumerate}
\end{realizationproblems}

\begin{algebraicstudy}
    For square matrices $A$ over $\ZG$,
$G$ a finite group, make a satisfactory
algebraic study of
the $\ZG[t]$-modules
$\cok (I-tA)$ and the associated $\ZG$-modules
$\cok (I-A)$. (The latter arise as
invariants of $G$-equivariant flow equivalence
\cite{BS05}.)
\end{algebraicstudy}


\begin{sufficiencyofinvariants}
  The following questions are open even for $G=\{e\}$.
\begin{enumerate}
\item
For \gps matrices, what invariants must be added to SSE-$\ZG$ to
imply SSE-$\Z_+G$?
\item
Prove or disprove: for $G$ nontrivial, every
SSE-$\ZG$ class of \gps matrices contains infinitey many
SSE-$\Z_+G$ classes.
\end{enumerate}
\end{sufficiencyofinvariants}

\appendix

\section{$G$-SFTs defined from matrices: left vs. right action}
\label{leftvsright}
In this section we describe how $G$ extensions of SFTs
are defined from matrices over $\Z_+G$, and the
corresponding classifying role of strong shift equivalence
of the matrices over $\Z_+G$ (SSE-$\Z_+G$).
In the process, we
 correct (see the Erratum \ref{erratum} below) an error in the corresponding
definition in \cite{BS05}.
Given $X\times G$, the map
 $g:(x,h)\mapsto (x,hg)$ defines
a right action of $G$ on $X\times G$, and the map
$g:(x,h)\mapsto (x,gh)$ defines
a left action of $G$ on $X\times G$.

There are corresponding notations for presenting a
$G$ extension.
Suppose $T:X\to X$ is a homeomorphism
and $\tau: X\to G$ is continuous. For the left action
on $X\times G$
we define the group extension
$T_{\ell,\tau}: X\times G
\to X\times G$ by $T_{\ell}: (x,h)\mapsto (T(x),h\tau(x))$.
For the right action we define
$ T_{r,\tau}: X\times  G \to X\times G$ by
$T_r: (x,h)\mapsto (T(x), \tau (x) h)$.
Each commutes with its associated $G$ action.

In the case of the left $G$ action,
continuous functions $\tau , \tau' $ from
$X \times G$ to $G$ are {\it cohomologous} if there is a
continuous $\gamma: X\to G$ such that for all $x$,
$\tau'(x) = \gamma^{-1}(x)\tau(x)\gamma(Tx)$.
In the case of the right action, the cohomology equation is
 $\tau'(x) = \gamma(Tx)\tau(x)\gamma^{-1}(x)$

Now suppose $A$ is square over $\Z_+G$.
The matrix $\overline A$ over $\Z_+$
is defined from $A$ by
applying the augmentation map
$\sum_g n_g g \mapsto \sum_g n_g$
entrywise. We view
 $\overline A$
as the adjacency matrix of a  directed graph.
If the set of edges from vertex $i$ to vertex $j$
is nonempty, label them by elements of $G$
to match $A(i,j)= \sum_g n_g g$: for each
$g$, exactly $n_g$ edges are labeled $g$.
Let $\tau_A  : X_{\overline A}\to G$ be the continuous function which
sends $x = \dots x_{-1}x_0x_1 \dots $ to the label of
the edge $x_0$, denoted $\ell (x_0)$.
We use $T_{\ell,A}$ and $T_{r,A}$ to
denote $T_{\ell,\tau }$ and $T_{r,\tau}$  with $\tau = \tau_A$.

In the case of the left $G$ action,
with $T$ the shift on $X_A$,
for
 the corresponding $G$ extension
$T_{\ell , A}$ defined on
$X_{\overline A} \times G$, for $n>0$ we have
\begin{align*}
T_{\ell}^n :  (x,h) \mapsto &\  (T^nx,h \tau_A(x) \cdots \tau_A(T^{n-1}x))
\\
=&\  (T^nx,h\ell(x_0) \cdots \ell (x_{n-1}) )\ .
\end{align*}
Here a weight $w=  \ell(x_0) \ell(x_1)\cdots \ell(x_{n-1})$ is
the product of the  labels along the edge-path
$x_0x_1\cdots x_{n-1}$. If $A^n(i,j)= \sum_g n_g g$, then
the number of edge paths with initial vertex $i$, terminal
vertex $j$ and weight $g$ is equal to $n_g$.
This is the connection of matrix and group extension
behind the following result
of Parry (see \cite[Prop. 2.7.1]{BS05}).
In the statement, $\tau_A \sim \tau_B \circ \varphi$
means there is a continuous $\gamma : X_{\overline A} \to G$ such that
$\tau_B (\varphi (x))  = \gamma^{-1} (x) \tau_A (x)\gamma (\sigma_A x)$. In the proposition we need only assume that $G$ is a discrete group, not necessarily finite. In this case, any continuous function into $G$ will then be locally constant.

\begin{prop}\label{prop_parry}
Let $G$ be a discrete group.
The following are equivalent for matrices
$A$ and $B$ over $\Z_+G$.
\begin{enumerate}
\item
$A$ and $B$ are SSE over $\Z_+G$.
\item
There is a homeomorphism
$\varphi\colon X_{\overline A}\to X_{\overline B}$
such that
$\varphi \sigma_{\overline A}= \sigma_{\overline B} \varphi$
and $\tau_A \sim \tau_B \circ \varphi$.
\item
The $G$-SFTs $T_{\ell ,A}$ and
 $T_{\ell ,B}$  are $G$-conjugate.
\end{enumerate}
\end{prop}

Explanation for all this is in \cite{BS05}-- after correction of the
following error.
\begin{erratum} \label{erratum}
In \cite[Sec. 2.4]{BS05},  the group extensions (skew products)
were
defined as extensions for the right $G$ action on $X\times G$.  They should
instead be extensions for the left $G$ action on $X\times G$.
Consequently two other changes should be made.
\begin{enumerate}
\item
In paragraph 2 of \cite[Sec. 2.7]{BS05},
``draw an edge from  $(g,i)$ to $(\ell (e)g,j)$"
should be ``draw an edge from $(g,i)$ to $(g\ell(e),j)$''.
\item
In the final sentence of paragraph 2 of
\cite[Sec. 2.7]{BS05},
``$(h,j) \mapsto (hg,j)$'' should be
''$(h,j) \mapsto (gh,j)$'' .
\end{enumerate}
\end{erratum}
\begin{remark} We record below some relations among
matrices and extensions.
We use $A'$ to
denote the
transpose of a matrix $A$;
if $A$ has entries in
$\Z_+G$,
we let
$A^{\textnormal{opp}}=A^\textnormal{o}$ be the
matrix defined by applying entrywise the map
$\sum_g n_g g \mapsto \sum_g n_g g^{-1}$.
(This map is an isomorphism from  $\ZG$ to its
opposite ring.)
\begin{enumerate}
\item
$(T_{\ell ,A})^{-1}$ and $T_{\ell, (A')^\textnormal{o}}$
are conjugate $G$ extensions.
\item
The $G$ extension
$T_{r,A}$ is conjugated to the  $G$
extension $T_{\ell ,A^{\textnormal{o}} }$,
by the map $(x,h)\mapsto (x,h^{-1})$.
(Note,  $(x,hg)\mapsto (x,(hg)^{-1})= (x,g^{-1}h^{-1})$ .)
\item
$T_{r,A}$ and $T_{r,B}$ are conjugate $G$ extensions
$\iff $ $A^{\textnormal{o}}$ and
$B^{\textnormal{o}}$ are SSE-$\Z_+G$.
\item
$A$ and $B $  SSE-$\Z_+G$
$\implies$
$(A')^{\textnormal{o}}$ and $(B')^{\textnormal{o}}$
are  SSE-$\Z_+G$. \\
(Note:  $A=RS, B=SR$
$\implies $
$(A')^{\textnormal{o}}=(S')^{\textnormal{o}}(R')^{\textnormal{o}},
(B')^{\textnormal{o}}=(R')^{\textnormal{o}}(S')^{\textnormal{o}}$.)
\item For $G$ nonabelian, for $A$ and $B$ SSE-$\Z_+G$: \\
$A'$ and $B'$ need not be SSE-$\Z_+G$;
$\ A^{\textnormal{o}}$ and $B^{\textnormal{o}}$ need not be SSE-$\Z_+G$.\\
(See Example \ref{oppex}).
\item
For $G$ nonabelian, for $T_{\ell,A}$ and $T_{\ell, B}$ conjugate
$G$-extensions: \\
$T_{r,A}$ and $T_{r, B}$ need not be conjugate
$G$-extensions.
\end{enumerate}
\end{remark}
\begin{example} \label{oppex}
Let $A \sim B$ mean $A$ and $B$ are SSE-$\Z_+G$.
We give an example here of $A\sim B$ with
$A^{\textnormal{opp}} \not\sim B^{\textnormal{opp}}$ and
$A'\not\sim B'$.
We use $G$ the group of permutations on $\{1,2,3,4\}$,
in which $gh$ is defined by
$(gh)(x)=g(h(x))$.
Let $M[x,y,z]$ denote a matrix $M$ with $M(1,2)=x, M(2,3)=y, M(3,1)=z$
and $M=0 $ otherwise. In $G$, define
$a=(143),
b=(123),c=(12)(34),d=(13)(24)$; then $abc=e $ and $a^{-1}b^{-1}c^{-1}
=d\neq e$.  Set $A=M[a,b,c]$ and $B=M[e,e,e]=B^{\textnormal{opp}}$.
 Then
$A\sim M[e,e,abc]=B$, but
$A^{\textnormal{opp}}=M[a^{-1},b^{-1},c^{-1}] \sim
M[e,e,a^{-1}b^{-1}c^{-1}] =M[e,e,d]$, and
$M[e,e,d]\not\sim B$ (e.g. by Proposition \ref{finitetracedata}).
Therefore $A^{\textnormal{opp}}\not\sim
B^{\textnormal{opp}}$. Similarly,
$B'\sim B$, and $A'\sim M[e,e,cab]=M[e,e,d]\not\sim B$.
\end{example}

\section{\gps  matrices and
shift equivalence}

\label{zgprimse}

{\bf Primitivity for matrices over $\mathbb Z G$.}

In this section,  $G$ is a finite group.
We will spell out some basic facts around
the regular representation of $G$,  our use of the
Perron Theorem and SE over $\Z_+G$.

Let $m=|G|$. Fix an enumeration of the elements of $G$,
$G=\{g_1, \dots ,g_m\}$, with $g_1=e$, the identity element.
If $x=\sum_i n_ig_i \in \Z G$, then its image under the augmentation
map is $\overline x= \sum_i n_i$.

\begin{definition} \label{augdefn}
For  vectors $v$ and matrices $M$
over $\ZG[t]$ (perhaps over just $\ZG$),
we define $\overline v$ and $\overline M$ by applying the
augmentation map entrywise,
$\sum_{s=0}^S \sum_g  m_{sg} g t^s \mapsto
\sum_{s=0}^S \sum_g  m_{sg}  t^s $ .
\end{definition}

\begin{notation} \label{zgtentrynotation}
Given a matrix $A$ over $\ZG$,
define $a_{ij} = A(i,j)$ and $a_{ijk} = A^k(i,j) $,  and let
$a_{ijkg}$ be the
integers such that
\[
A^k(i,j) = a_{ijk}=\sum_{g} a_{ijkg}\, g \   .
\]
Define
$\overline{a}_{ij}=\sum_{g} a_{ij1g}$, i.e.,
$\overline A (i,j) = \overline{a}_{ij}$.
The uppercase - lowercase correspondence above producing $a$ given
$A$ may be
used for other letters as well.
\end{notation}


Let $e_i$ denote the size $m$ column vector whose $i$th entry
is 1 and whose other entries are zero. Define an isomorphism
of additive groups $p: \Z G \to \Z^m$ by the rule
$\sum_i n_i g_i\mapsto \sum_i n_ie_i$.  We carry over the usual
partial order on $\Z^m$: for $x= \sum_i n_ig_i$ we say
$x\geq 0$ if $n_i\geq 0$ for all $i$, and we write
$x\gg 0$ if $n_i> 0$ for all $i$. When we use an order relation
for vectors or matrices, we mean that it holds entrywise.
For example, $x\gg 0$ in $\Z G$ if and only if
$p(x)>0$ in $\Z^m$. We also carry over the usual notion of
convergence in $\Z^m$:
a sequence of elements $x^{(k)}=
\sum_i n^{(k)}_i g_i $ converges to $x=\sum_i n_ig_i $ iff
$\lim_k  n^{(k)}_i = n_i $ for each $i$. Convergence of vectors
or matrices over $\Z G$ is by definition entrywise convergence.

For $1\leq r\leq m$,
define $m\times m$ permutation matrices $P_r, Q_r$ by
the rules
\begin{align*}
P_r(i,j)\ &= \ 1 \ \ \text{ iff } \ \ g_rg_j=g_i \\
Q_r(i,j)\ &= \ 1 \ \ \text{ iff } \ \ g_jg_r=g_i \ .
\end{align*}
Then $P_r(p(g_j)) = p(g_rg_j)$ and
$Q_r(p(g_j)) = p(g_jg_r)$. The map $g_r\mapsto P_r$ is the
regular representation of $G$ given by its action on itself
by multiplication from the left; similarly for $Q_r$ and
right multiplication.
For $x=\sum_j n_j g_j\in \Z G$, we similarly define $\rho (x)$
to be the $m\times m$ matrix over $\Z$ which presents
multiplication by $x$ from the left. That is, the following
diagram commutes, with $\rho (x) = \sum_j n_jP_j$ :
\[
\begin{CD}
\Z G    @>x >>    \Z G \\
@VpVV      @VVpV \\
\Z^m @>\rho(x)>> \Z^m
\end{CD}
\qquad \quad
\begin{CD}
y       @>>>    xy  \\
@VVV      @VVV \\
p(y)    @>>>    \rho(x) p(y)
\end{CD}
\]
Column 1 of the matrix $\rho(x)$ is
$p(x)=\left(\begin{smallmatrix} n_1\\ \vdots \\ n_m\end{smallmatrix}\right)$,
since $p(x)=p(xe)=\rho(x)p(e)=\rho (x) e_1$.
For each $j$, column $j$ of $\rho(x)$ is
$Q_j \left(\begin{smallmatrix} n_1\\ \vdots \\ n_m\end{smallmatrix}\right)$,
since
column $j$ of $\rho(x)$ equals
\[
\rho(x)e_j
=\rho(x)p(g_j)=p(xg_j)=Q_jp(x)\  .
\]

Now suppose $A$ is $\ell\times n$ over $\Z G$.
Define an $\ell m\times nm$ matrix $\widetilde A$, with a block form of
$m\times m$ blocks, in which the $ij$ block is
$\rho (a_{ij})$.
If $A,B$ over $\Z G$ have compatible sizes for
matrix multiplication, then $\widetilde A \widetilde B =
\widetilde{AB}$.
Letting  $\kappa$ be defined as in Definition
\ref{kappadefn}, we pause
to record some  facts used
in Section \ref{sec:fge} to discuss
the periodic data \eqref{perdat}.
\begin{proposition} \label{finitetracedata}
Let $G$ be a finite group, with $m=|G|$.
Suppose $A$ is an $n\times n$ matrix over $\ZG$.
 Let $\eta (t)\in \Z[t]$ be the characteristic polynomial of $\widetilde A$.
Then $\eta (A)=0$, and
 \begin{enumerate}
\item
the finite sequence
$(\trace (A^k))_{1\leq k \leq mn}$, determines
$(\trace(A^k))_{1\leq k < \infty} $ .
\item
the finite sequence
$(\kappa (\trace(A^k)))_{ 1\leq k \leq mn}$ determines
$(\kappa (\trace(A^k)))_{1\leq k < \infty} $.
\end{enumerate}
If $A$ and $B$ are matrices SSE over $\ZG$, then
$(\kappa (\trace(A^k))_{ 1\leq k <\infty} =
(\kappa (\trace(B^k))_{ 1\leq k < \infty}$.
\end{proposition}
\begin{proof}
$\eta(A)=0$ because $A\mapsto \widetilde A$
defines an embedding of the ring of $n\times n$ matrices
over $\Z G$ into the ring of $nm\times nm$ matrices over
$\Z$. The coefficients of $\eta$ are determined by
the finite sequence
$(\trace (\widetilde A^k))_{ 1\leq k \leq mn}$,
which equals
$(m\sum_i a_{iike})_{1\leq k \leq mn} $, which is determined by
$(\kappa (\trace (A^k))_{1\leq k \leq mn} $.
The  claims (1,2) then follow because
$\eta(A)=0$ gives integers $c_1, \dots , c_{nm}$ such that
$\trace (A^k) = c_1 \trace (A^{k-1}) + \dots + c_{nm}
\trace (A^{k-nm})$ for all $k>mn$.
It suffices to prove  the final claim
in the case
$A=RS, B=SR$ for matrices $R,S$ over $\ZG$. Then
\[
\trace (RS)= \sum_{i,j,g,h} r_{ij1g}s_{ji1h}gh \quad
\text{and} \quad
\trace (SR) = \sum_{i,j,g,h} s_{ji1h}r_{ij1g}hg  \ .
\]
Because $gh = h^{-1}(hg)h$, it follows that
$\kappa (\trace (RS)) =
\kappa (\trace (SR)) $ .  For $k>1$, we have
$A^k = (A^{k-1}R)S$ and
$B^k = S(A^{k-1}R)$ . The conclusion follows.
\end{proof}


\begin{definition} \label{defngprimitive}
  For a matrix $A$
  over $\R G$
(e.g., $A$ over $\ZG$),
we say $A$ is \gps  if
$A$ is square, $A\geq 0$ and,
for some $k>0$,
$A^k\gg 0$.
\end{definition}
 Clearly $A$ is \gps  if and only
 if $\widetilde A$ is primitive,
 since  $A\gg 0$ is equivalent to $\widetilde A >0$.
  (For  an example, consider
$G=\Z /2$, $g\neq e$ and $A=(5g)$, giving $\widetilde A =
\left(\begin{smallmatrix}0&5\\5&0\end{smallmatrix} \right)$;
here $\overline A$ is primitive but $A$ is not \gps .)
The spectral radius of a real matrix $M$ is denoted $\lambda_M$.
The matrices $\overline A$ and $\widetilde A$ have the same
spectral radius.

\begin{definition} \label{specradiusdefn}
Let $G$ be a finite group. The spectral radius $\lambda_A$
of a square matrix $A$ over $\ZG$ is defined to be
$\lambda_{\overline A}= \lambda_{\widetilde A}$. The spectral
radius $\lambda_A$ of a square matrix $A$ over $\ZG[t]$ is defined to
be the spectral radius of $A^{\Box}$.
\end{definition}
Naturally, for $A$ square over $\ZG$, we have
$\lambda_A = \varlimsup_k \max_{i,j,g} |a_{ijkg}|^{1/k}$ .

\begin{theorem}\label{zgperron} Suppose $G$ is a finite group,
$G=\{g_1,\dots ,g_m\}$ with $g_1=e$,
the identity element of $G$. Suppose $A$ is an $n\times n$
matrix over $\Z_+ G$ such that its augmentation $\overline A$
is irreducible. Let $\lambda=\lambda_A$.
For  $i$ in $\{1,\dots ,n\}$, set
$H_i=\cup_k \{g\in G: a_{iikg}>0\}$.
Then the following statements are true.
\begin{enumerate}
  \item
The sets $H_i$ are
conjugate subgroups of $G$.
\item
The following are equivalent.
\begin{enumerate}
\item
$\widetilde A$ is primitive.
\item
$A$ is \gp .
\item
Let $\overline{\ell}, \overline r$
denote  positive left and right eigenvectors
of $\overline A$
such that
$\overline{\ell}\overline{r}= (1)$
(these vectors exist because
$\overline A$ is irreducible). Then
\[
\lim_k \Big(\frac 1{\lambda}A\Big)^k\ = \ (g_1+\cdots +g_m)
 \frac 1m\overline r \overline{\ell} \ .
\]
\item With the notation $\trace(A^k)= \sum_g \tau_{kg} g$,
the following conditions hold:
\begin{enumerate}
\item
There are relatively prime $j,k$ such that
$\tau_{ke}>0$ and $\tau_{je}>0$.
\item There exists   $i$ such that
$H_i = G$.
\end{enumerate}
\end{enumerate}
\item
If $G$ is abelian
and $\overline A$ is irreducible, then the polynomial $\det (I-tA)$
determines whether $A$ is \gps.
\end{enumerate}
\end{theorem}

\begin{remark}{\normalfont
It follows from the Perron theorem that
the convergence in (3) above
 is exponentially fast. }
\end{remark}

\begin{proof}[Proof of Theorem \ref{zgperron}]
(1) Given $i$, there exists a diagonal matrix
$D$, with each diagonal entry an element of $G$, such
that $D^{-1}AD$ has all entries in $H_i$
\cite[Proposition 4.4]{BS05}. As in
\cite{BS05}, it follows easily that the $H_i$
are conjugate subgroups of $G$.

(2)
(a) $\iff $ (b) This was part of the paragraph
before the theorem.


(b) $\implies$ (c)
Let $u$ denote $g_1+\cdots +g_m$.
The augmentation matrix $\overline A$ is primitive,
because $A$ is \gpsperiod\  Therefore
$((1/\lambda )\overline A)^k$ converges  to
$\overline{r}\overline{\ell}$. Define size $n$ vectors over
$\R_+G$ by setting $\ell = u \overline{\ell}$ and
$r = u \overline r$. If $x=\sum_in_ig_i \in \Z G$, then
$xu  = (\sum_i n_i)u= ux$. Therefore
\[
Ar = Au\overline r = u\overline A \overline r = u\lambda \overline r
=\lambda r
\]
and likewise $\ell A = u\lambda \overline{\ell} =\lambda \ell$ .
These eigenvectors lift to eigenvectors
$\widetilde{\ell}, \widetilde r$
of $\widetilde A$.
Explicitly,  $\widetilde{\ell} = (\widetilde{\ell}_1, \dots ,
\widetilde{\ell}_n)$ in which
$\widetilde{\ell}_j $ is the size $m$ row vector
$p(u\overline{\ell_j})$; every
entry of $\widetilde{\ell}_j$
equals $\overline{\ell}_j$.
Likewise,
every
entry of $\widetilde{r}_j$
equals $\overline{r}_j$.
We have
$(\widetilde{\ell}\widetilde r) = m(\overline{\ell}\overline{r})$.
Only now do we appeal to the primitivity of $\widetilde A$,
which guarantees
\[
\lim_k \Big(\frac 1{\lambda}\widetilde A\Big)^k \ =\  \frac 1m
\Big(\widetilde r \widetilde{\ell}\Big) \ .
\]
Translated back to $A$, this becomes
\[
\lim_k
\Big(\frac 1{\lambda} A\Big)^k \ =\
(g_1 + \cdots + g_m)
\Big(
\frac  1m
\overline r \overline{\ell}
\Big)\ .
\]

(c) $\implies$ (d) Obvious.

(d) $\implies$ (b) The subgroups $H_i$ are conjugate,
so (d) implies that
$H_i=G$ for every $i$.
Now suppose $j,k$ are
relatively prime with   $\tau_{ke}>0$ and $\tau_{je}>0$.
Pick indices $y,z$ such that
$(A^j)_{yye}>0$ and $(A^k)_{zze}>0$. If $y=z$ then for all large
$M$ we have $a_{yyMe}>0$, and because $H_y=G$ we have
for all $g$ and all large $M$ that $a_{yyMg}>0$.
It then easily follows from the irreducibility of $\overline A$ that
$A$ is \gp.

So suppose $y\neq z$. Because $\overline A$ is irreducible,
we may choose
 integers $s,s'$ such that $\overline A^s(y,z)>0$
and $\overline A^{s'}(y,z)>0$.
There are corresponding paths $\pi, \pi'$
in the labeled graph with adjacency matrix $A$,
say with  weights $g$ and $g'$.
Let $\pi^*$ be a path from $q$ to $q$ with  length $k$ and
weight $e$.
The concatenation $\pi \pi'$
is a path of length $s+ s'$ and weight $gg'$
from $y$ to $y$.
Pick $r$ such that $(gg')^r=e$. Then the path
$(\pi \pi')^{jr-1}\pi \pi^* \pi'$ is a path from
$y$ to $y$ of weight $e$ and length $jr+k$,
which is relatively prime to $j$. The argument of the
last paragraph then  applies to show $A$ is \gps .


(3) Suppose $G$ is abelian. In this case the conjugate
groups $H_i$ are equal and must equal
$\cup_k \{g: \tau_{kg}>0\} $. Thus $A$ is $G$-primitive
if and only if for some relatively prime $j,k$ we have
$\tau_{ke} \gg 0$ and $\tau_{je} \gg 0$. This is easily
checked with $\det(I-tA)$, which
constructively determines $(\tr (A^k))_{k\in \N}$.

\end{proof}

\begin{corollary} \label{gmixa}
A matrix $A$ over $\Z_+G$ defines a mixing $G$-extension if and only if
$A$  is essentially \gps .
\end{corollary}

\begin{proof}
The $G$ extension defined by $A$ is a SFT defined by
$\widetilde A$, and therefore is topologically mixing
if and only if $\widetilde A$ is essentially primitive as
a matrix over $\Z_+$.
Therefore
the corollary follows from the equivalence of (1) and (2)
in  Proposition \ref{zgperron}.
\end{proof}

{\bf Polynomial matrices.}
Given $A$ over $t\Z_+G[t]$, we have
$\sum_n \trace (A^n) = \sum_n \trace ((A^{\Box})^n))t^n$,
and for $G$ abelian
$\det (I-A) = \det(I-tA^{\Box})$. By Theorem \ref{zgperron},
this data determines whether $A$ is \gp.

 We will need the following
consequence of Theorem \ref{zgperron}.

\begin{proposition}\label{looprate}
Suppose $A=(a)$ is a $1\times 1$
matrix over $t\Z G_+[t]$ with
$A^{\Box}$ \gpsperiod\
Let $(\alpha_k)$ be the sequence of elements from
$\Z G$ such that
\[
\sum_{j=1}^{\infty} a^j \ = \ \sum_{k=1}^{\infty} \alpha_k t^k\ .
\]
Then there is a positive real number $c$ such that
\[
\lim_{k\to \infty} \frac 1{(\lambda_A)^k}\, \alpha_k
\ = \
c(g_1+\dots + g_m) \ .
\]
\end{proposition}

\begin{proof}
The matrix $A^{\Box}$ is the adjacency matrix of a loop
graph $\mathcal G$ with base vertex 1. Let $a = \sum_{k,g}a_{kg}
gt^k$, with the $a_{kg}$ in $\Z_+$.
Then in $\mathcal G$, for every positive coefficient $a_{kg}$,
there are $a_{kg}$ first return loops to 1 of length $k$ and weight
$g$. The  return loops to $1$ are formed from all concatenations
of first return loops. Under concatenation, lengths add and
weights multiply. Consequently, for all $k$, $(A^{\Box})^k(1,1)
= \alpha_k$. The proposition is then a consequence of
Theorem \ref{zgperron}.
\end{proof}


In the rest of this section, we check that two
standard results for SFTs carry over to $G$-SFTs.
The main interest of the next proposition is
$(1)\iff (2)$. The proof is an adaptation of
the proof  of Kim and Roush in the $\Z$
case (see
\cite[Section7.5]{LindMarcus1995} or \cite{S33}).

\begin{proposition}\label{seetc}
Suppose $G$ is a finite group, $\sss = \Z_+G$
or  $\sss = \Z G$,
and $A$ and $B$ are square matrices over $\sss$.
Then the following are equivalent.
\begin{enumerate}
\item
$A$ and $B$ are SE over $\sss$.
\item
$A^n$ and $B^n$ are ESSE over $\sss$ for all large $n$.
\item
$A^n$ and $B^n$ are SE over $\sss$ for all large $n$.
\item
Let $A$ be $n_1\times n_1$ and let $B$ be $n_2\times n_2$.
Let $n=\max \{n_1,n_2\}$ and let $m =|G|$.
Then there exists $k$ such that $A^k$, $B^k$ are SE over
$\sss$ and $k\equiv 1\ \textnormal{mod } ((mn)^2)!\ .$
\end{enumerate}
\end{proposition}
\begin{proof}
Clearly $(1)\implies (2) \implies (3)\implies (4)$.
Now, to show
$(4)\implies (1)$, assume $(4)$.
Then we have $\ell \in \mathbb N$,
$k\equiv 1\ \textnormal{mod } ((mn)^2)!$
 and matrices $U,V$ over $\sss$
such that the following hold:
\[
(A^k)^{\ell} = UV\ ,    \quad (B^k)^{\ell} =VU \ , \quad
A^kU = U B^k\ , \quad B^kV = VA^k\ .
\]
For $i\geq n$ and $k\geq n$ define $U_i= A^iU$ and $V_j=B^jV$.
Then
\[
(A^k)^{\ell +i+j} = U_iV_j\ ,    \quad (B^k)^{\ell +i+j} =V_jU_i \ , \quad
A^kU_i = U_i B^k\ , \quad B^kV_j = V_jA^k\ .
\]
Via the map $\ZG \to \Z^m$ discussed earlier, this gives
a shift equivalence of matrices over $\sss$,
\[
(\widetilde A^k)^{\ell +i+j} = \widetilde{U_i}\widetilde{V_j}\ ,
\quad (\widetilde B^k)^{\ell +i+j} =\widetilde{V_j}\widetilde{U_i} \ ,
\quad \widetilde A^k\widetilde{U_i} = \widetilde{U_i}\widetilde B^k\ ,
\quad \widetilde B^k \widetilde{V_j} = \widetilde{V_j}\widetilde A^k\ .
\]
Choose $i$ such that $\ell + i +j
\equiv 1\ \textnormal{mod } ((mn)^2)!$.
It suffices to show that  the two intertwining equations then
hold with $k$ replaced by 1 (as this translates to the equations
holding with the $\ \widetilde{ }\ $ decorations removed). Let $r=k(\ell +i+j)$.

Consider the intertwining equation for $U_i$. The matrix $A$ is
$mn_1 \times mn_1$, and  $\mathbb C^{mn_1}$ is the direct sum
of the kernel $\K_A$ and the image $W_A$ of $A^{mn}$. Because
$i\geq mn$, restricted to $\K$ we have
$\widetilde A\widetilde{U_i} = \widetilde{U_i}\widetilde B= 0$ .
Also, $U_i$ maps $W_A$ isomorphically to $W_B$, the image of
$B^{mn}$. An invariant Jordan subspace of $A$ for eigenvalue
$\alpha\neq 0$ is mapped by $U_i$ to
an invariant Jordan subspace of $B$ for eigenvalue
$\beta\neq 0$, such that
$\alpha/\beta$ is a root of unity $\xi$ such that
$ \xi^r =1$. Because $\xi$ is in the
number field generated by $\alpha$ and $\beta$,
$\xi$ is a $q$th root of unity with $q\leq (mn)^2$,
and therefore $q$ divides $((mn)^2)!\, $. Consequently
$\xi^r =\xi$ and $\xi =1$. It follows that
$\widetilde A\widetilde{U_i} = \widetilde{U_i}\widetilde B$ .
The same argument works for the other intertwining
equation.
\end{proof}

\begin{proposition}\label{eventualconjugacy}
Suppose $G$ is a finite group
and $A$ and $B$ are square matrices over $\Z_+G$.
Then the following are equivalent.
\begin{enumerate}
\item
The $G$-SFTs $\sigma_A,\sigma_G$ are eventually conjugate.
\item
The matrices $A$, $B$ are SE over $\Z_+G$.
\end{enumerate}
\end{proposition}
\begin{proof}
Clearly $(2)\implies (1)$. Also, $(2)$ implies $A^n$ and $B^n$
are SE over $\Z_+G$ for all large $n$, and this implies $(1)$
by Proposition \ref{seetc}.
\end{proof}

\begin{proposition}\label{primitiveeventual}
Suppose $A,B$ are \gps . Then
the following are equivalent.
\begin{enumerate}
\item
$A$ and $B$ are SE over $\Z_+G$.
\item
$A$ and $B$ are SE over $\ZG$.
\end{enumerate}
\end{proposition}
\begin{proof}
Assuming (2), it suffices to prove (1).
We have matrices $U,V$ over $\ZG$ giving
the assumed shift equivalence of $A,B$.
Then $\widetilde U, \widetilde V$ give
a shift equivalence of
$\widetilde A, \widetilde B$.
Perhaps after
replacing $U,V$ with $-U,-V$ we
have that $U$ takes positive left/right
eigenvectors of $\widetilde A$ to positive
left/right eigenvectors for $\widetilde B$,
and likewise for $V$.  It follows
from the spectral gap given by
primitivity that for large $k$,
the matrices
$\widetilde A^kU$ and $\widetilde B^kV$ are strictly positive.
They give an SE over $\Z_+$ of
$\widetilde A,\widetilde B$ and consequently produce
an SE over $\Z_+G$ of
$A,B$.
\end{proof}

\section{$\nkone ( \ZG )$}\label{sec:nk1}
 Let $\Rcal$ be a  ring (always assumed to be unital).
In this appendix, we give background on the group $\nkone (\rr )$,
especially for $\rr = \ZG $, with $G$ a finite group.

The first algebraic $\K$ group is defined by $\text K_{1}(\Rcal) =
\GL(\Rcal)/\EL(\Rcal)$, where $\GL(\Rcal) = \varinjlim \GL_{n}(\Rcal)$ and
$\EL(\Rcal) = \varinjlim \EL_{n}(\Rcal)$, $\EL_{n}(\Rcal)$ the elementary
matrices of size $n$.
If $\Rcal$ is also commutative, then the determinant map $\det:\Rcal
\to \Rcal^{\times}$ is a split surjection, and gives a decomposition
$\K_{1}(\Rcal) \cong \SK_{1}(\Rcal) \oplus \Rcal^{\times}$, where
$\SK_{1}(\Rcal)=\ker (\det )$, and $\Rcal^{\times}$ denotes the group of units in $\Rcal$.\\
\indent The group $\nkone(\Rcal)$ is defined to be $\ker(\K_{1}(\Rcal[t]) \stackrel{t \to 0}\to\ \text{K}_{1}(\Rcal))$. The exact sequence $0 \to t\Rcal[t] \to \Rcal[t] \stackrel{t \to 0}\to \Rcal \to 0$ is split on the right, giving a decomposition $\K_{1}(\Rcal[t]) \cong \nkone(\Rcal) \oplus \K_{1}(\Rcal)$.  Higman's trick shows that $\nkone(\Rcal)$ is generated by elements of the form $[I-tN]$, with $N$ nilpotent. If $R$ is reduced (has no non-trivial nilpotents), then one also has $\nkone(\Rcal) \subset \SK_{1}(\Rcal[t])$.


For any ring $\rr$,  $\nkone(\Rcal) $ either is trivial or is not
finitely generated
as a group.
For many rings $\rr$, $\nkone(\rr )=0$.
For any regular Noetherian ring $\rr$,
$\nkone(\rr )=0$.
For example, a polynomial ring $\rr [x_1, \dots , x_n]$ is regular
Noetherian
if
$\rr$ is a field, $\Z$, a Dedekind domain or any ring with finite
global dimension.
See \cite{Rosenberg1994,WeibelBook} for all this and
more.
However,  if $G$ is a non-trivial finite group,
then $\mathbb{Z}G$ is not regular, and in general the computation of
$\nkone (\mathbb{Z}G)$ is difficult.
If $G$ is any finite group of square-free order, then
$\nkone(\mathbb{Z} G) = 0$ \cite{Harmon1987}.
In \cite{Weibel2009}, it is shown that $\nkone(\mathbb{Z}[\mathbb{Z}/2
\oplus \mathbb{Z}/2])$, $\nkone(\mathbb{Z}[\mathbb{Z}/4])$, and
$\nkone(\mathbb{Z}[D_{4}])$, where $D_{4}$ denotes the dihedral group
of order 8, are all non-zero. In fact, both
$\nkone(\mathbb{Z}[\mathbb{Z}/2 \oplus \mathbb{Z}/2])$ and
$\nkone(\mathbb{Z}[\mathbb{Z}/4])$, as abelian groups, are isomorphic
to a countably infinite direct sum of copies of $\mathbb{Z}/2$, while
$\nkone(\mathbb{Z}[D_{4}])$ is a quotient of a direct sum of a
countably infinite free $\mathbb{Z}/4$ module and a countably infinite
free $\mathbb{Z}/2$ module \cite{Weibel2009}.\\
\indent While the situation for $\mathbb{Z}[G]$ with $G$ a general finite group is complicated, more is known for finite abelian groups. It follows from Theorem 3.12 in \cite{MartinThesis} together with Theorem 1.4 from \cite{Weibel2009} that $NK_{1}(\mathbb{Z}[\mathbb{Z}/p^{n}]) \ne 0$ for $n \ge 2$ with $p$ prime\footnote{This is also proved in \cite{SchmiedingNK1FinAbGrps}}. This taken together with Theorem 3.6 in \cite{MartinThesis} then implies that for a general finite abelian group $G = \bigoplus_{i=1}^{n}\mathbb{Z}/p_{i}^{k_{i}}$, $NK_{1}(\mathbb{Z}[G])$ is non-zero if one of it's $p$-primary cyclic components has $p$-rank greater than 1, i.e. $k_{i} \ge 2$ for some $1 \le i \le n$.

For any ring $\rr$ and finite group $G$,
$\nkone(\Rcal G$) is a torsion group
\cite{HambletonLueck2007,WeibelBook}.
In fact, \cite[Theorem A]{HambletonLueck2007} shows that the order of
every element of $\nkone(\Rcal G)$ is some power of $|G|$, whenever
$\nkone(\Rcal) = 0$.
(For $\Rcal = \mathbb{Z}$, and other rings, this is a result of
 Weibel.)
In particular, if $P$ is a finite $p$-group, then every element of $\nkone(\mathbb{Z}P)$ has $p$-primary order \cite{HambletonLueck2007}.

\begin{proposition} \label{sk1fact}
Suppose  the ring $\rr$ is  commutative and reduced (i.e.,
has no nonzero nilpotent
element).
Then the following hold.
\begin{enumerate}
\item
Let $N$ be a nilpotent matrix over $\rr$. Then $\trace  (N^k)=0$
for all $k$ in $\N$.
\item
  $\nkone(\Rcal) \subset \SK_{1}(\Rcal[t])$.
\end{enumerate}
If  $G$ is a finitely generated abelian group, then
$\nkone(\mathbb{Z}G) \subset \SK_{1}(\mathbb{Z}G[t])$.
\begin{proof}
(1)
  Suppose $N$ is nilpotent
  with $\trace (N^{\ell}) = \alpha \neq 0$.
  Without loss of generality,
suppose $\trace (N^j) = 0$ for $j>{\ell}$.
Set $M=N^{\ell}$ and suppose $M^J=0$.
 Let $\det (I-tM) = 1 -c_1t -c_2t^2 - \cdots $.
Then $c_1 = \alpha $ and for $k>1$,
\[
\trace (M^k)= kc_k + \sum_{1\leq j < k} c_j
\trace (M^{k-j})
= kc_k + c_{k-1}\trace (M) \ .
\]
By induction, $(k!)c_k = (-1)^{k+1} \alpha^k$,  for all $k$ in
$\N$. Since $\det(I-tM)$ is a polynomial,
$\alpha$ is nilpotent, a contradiction.

(2)
An element of  $\nkone(\Rcal)$
contains a matrix of the form $I-tN$, where $N$ is nilpotent over
$\rr$.
Since $I-tN$ is invertible,
$\det(I-tN)$ must be a unit in the polynomial ring $\Rcal[t]$.
Because $\Rcal$ is commutative and reduced,
 the only units in $\Rcal[t]$ are
degree zero polynomials, 
and therefore  $\det(I-tN) = 1$.



For a finitely generated abelian group $G$,
it follows from a
theorem of Sehgal \cite[page 176]{Sehgal} that $\mathbb{Z}G$ has no
nilpotent elements.
\end{proof}
\end{proposition}

For a ring $\rr$, the reduced nil group $\text{Nil}_0(\rr )$ is an
abelian group which may be presented by generators and relations
as follows. The generator set is the set of nilpotent matrices. The
relations are $A = A \oplus 0$ (where $0$ is any square zero matrix
and $A$ is nilpotent); $A =U^{-1}AU$ ($A $ nilpotent, $U$ invertible
over $\rr$);
and for any block
matrix with $A,B$ square nilpotent,
\[
A + B = \begin{pmatrix} A & X \\ 0 & B \end{pmatrix} \ .
\]
An important corresondence in $\K$-theory
is that the map $N\mapsto I+tN$ (defined for $N$ nilpotent)
induces a well defined isomorphism from
$\text{Nil}_0(\rr)$ to $\nkone (\rr )  $.\\

\textbf{Explicit examples over $\mathbb{Z}G$}

Below we give some explicit examples of elements in
$\nkone$ of certain integral group rings.

\begin{example} \label{nk1example}
We give a $2\times 2$ matrix $M$ which represents a
nontrivial element of $\nkone(\ZG)$, for the cyclic group $G=\Z /4\Z$.
(The justification in \cite{SchmiedingExamples}
for the example is a nontrivial and
computer-assisted exercise in $\K$-theory.)
We let $\sigma$ be a generator of $G$ and set
$$M=
\begin{pmatrix} 1-a & -b \\
-c & 1-d  \end{pmatrix}$$
with
\begin{align*} a&= (1-\sigma^{2})(x-2x^2+2x^3-\sigma + x\sigma +x^{2}
  \sigma) \\
b&=(1-\sigma^{2})(1+2x-x^{2}-x^{3}-2x^{4}+\sigma-x\sigma -2x^{2}\sigma-3x^{3}\sigma+2x^{4}\sigma) \\
c&=(1-\sigma^{2})(-1+2x-5x^{2}+7x^{3}-3x^{4}+2x^{5}-\sigma +2x\sigma -2x^{3}\sigma +3x^{4}\sigma -2x^{5}\sigma) \\
d&= (1-\sigma^{2})(2+x-2x^{2}-4x^{4}-2x^{5}+\sigma -3x\sigma
-x^{2}\sigma -4x^{3}\sigma +6x^{4}\sigma -4x^{5}\sigma +4x^{6}\sigma)
\ .
\end{align*}
Because entries of the $2\times 2$ matrix $M$ have maximum degree
6, we can systematialy produce from $M$ a $12\times 12$ nilpotent matrix $N$
which is nontrivial in  $\text{Nil}_0(\ZG )$. We could work harder to
reduce the $12\times 12$ size a bit,
but we do not know how to produce a small nilpotent matrix
nontrivial in
 $\text{Nil}_0(\ZG )$.
\end{example}

\begin{example}
One could ask for an explicit example of two
\gps  matrices over $\Z_+ G$
with $G$ abelian which
are shift equivalent but not strong shift equivalent over $\Z_+G$
(and thus present nonisomorphic mixing group extensions).
We don't know small matrix examples for this, because we
don't know small examples of nilpotents nontrivial in
$\nkone(\ZG)$. We can do a
bit better with polynomial matrix presentations.
With $G=\Z/4\Z$ and $a,b,c,d$ from Example \ref{nk1example}
and $e,f$ elements of $\Z_+G[x]$,
consider the $4\times 4$ matrix
\begin{align*}
&\begin{pmatrix}
1 & 0 & 0 & 0 \\
0 & 1 & 0 &  0 \\
1 & 0 & 1 & 0 \\
1 & 0 & 0 & 1
\end{pmatrix}
\begin{pmatrix}
1 & 0 & 0 & 0 \\
0 & 1 & -1 & -1 \\
0 & 0 & 1 & 0 \\
0 & 0 & 0 & 1
\end{pmatrix}
\begin{pmatrix}
e & f & 0 & 0 \\
e & f & 0 & 0 \\
0 & 0 & a & b \\
0 & 0 & c & d
\end{pmatrix}
\begin{pmatrix}
1 & 0 & 0 & 0 \\
0 & 1 & 1 &  1 \\
0 & 0 & 1 & 0 \\
0 & 0 & 0 & 1
\end{pmatrix}
\begin{pmatrix}
1 & 0 & 0 & 0 \\
0 & 1 & 0 &  0 \\
-1 & 0 & 1 & 0 \\
-1 & 0 & 0 & 1
\end{pmatrix}
\\
& =
\begin{pmatrix}
e-2f  & f & f & f \\
e-2f +(a+b+c+d) & f & f-(a+c) & f-(b+d) \\
e - (a+b) & 0 & f-c & f-d \\
e - (c+d) & 0 & f-a & f-b
\end{pmatrix}
:= L \ .
\end{align*}
\end{example}

Let $K = \left( \begin{smallmatrix} e&f \\ e&f \end{smallmatrix}
\right)$.
Choosing $f$, and then $e$, with sufficiently large coefficients,
one has $K$ and $ L$  over  $\Z_+G[t]$
such that $K^{\Box}$ and $L^{\Box}$ are
\gps  matrices.
Because
$I-K$ and $I-L$  are not $\EL (\ZG[t])$ equivalent,
$K^{\Box}$ and $L^{\Box}$ are not
SSE over $\ZG$, and therefore the  associated group extensions
cannot be isomorphic. However,
$K^{\Box}$ and $L^{\Box}$ are shift equivalent over
$\ZG$ and therefore (since they are \gp )
shift equivalent over $\Z_+G$, by \ref{primitiveeventual}.
%

\bibliographystyle{plain}
\bibliography{mbssbib}

\section{Corrections} 
The content preceding this appendix is essentially
the content of the paper as accepted by Ergodic Theory and
Dynamical Systems  (doi:10.1017/etds.2015.87), 
 before processing by the publisher.
Two corrections should be made to the paper:
\begin{itemize}
  \item 
The bijection
of Theorem \ref{sseclassif}(2)
is not to $\text{NK}_1(\mathcal{R} )$, but to a certain quotient
group $\text{NK}_1(\mathcal{R} )/E(A, \mathcal{R})$.
\item
In Theorem \ref{zginfinite} there should be added the
  hypothesis that the elementary stabilizer $E(A,\mathcal{R})$
  is trivial. (This is know to hold if $A$ is invertible
  over $\mathcal R$, and in some other cases \cite{BoSc1}.)
\end{itemize}
The 
``elementary stabilizer'' $E(A, \mathcal{R})$ is defined to be 
the subgroup of elements $[U]$ in $\text{NK}_1(\mathcal{R} )$
such that there exists an elementary matrix $E$ such that
$U(I-tA) = (I-tA)E$. 

With these changes, 
  the theorems and proofs remain correct.
Note, the revised  Theorem  \ref{zginfinite}  
  still provides for every finite group $G$ with
nontrivial $NK_1(\Z G)$ many cases in
  which the answer to Parry's
  question is decisively no.

  We give now a little more context.
  
  Theorem \ref{sseclassif} 
states   a  result claimed in
the version 1 arXiv post of 
\cite{BoSc1}; this quoted result was corrected in
the version 2  post  of 
\cite{BoSc1} (and in the paper itself, to appear in Crelle's Journal),
where we proved for every commutative ring $\mathcal{R}$ that
\[
\bigcup_A E(A,\mathcal{R}) = NSK_1(\mathcal{R})  \ .
\]

\end{document}